\documentclass[leqno,12pt]{amsart}
\usepackage{amssymb}
\usepackage{amsmath}
\usepackage{enumerate}
\usepackage{amsfonts}
\usepackage{hyperref}
\usepackage{soulutf8}
\usepackage{tikz}

\hypersetup{
    colorlinks=true,
    linkcolor= blue,
    citecolor =cyan,
    urlcolor = teal,
}

\newtheorem{theorem}{Theorem}[section]
\newtheorem{corollary}[theorem]{Corollary}
\newtheorem{lemma}[theorem]{Lemma}
\newtheorem{proposition}[theorem]{Proposition}

\numberwithin{equation}{section}

\begin{document}
\title[Littlewood-Paley square functions in the Dunkl setting]{Upper and lower bounds for Littlewood-Paley square functions in the Dunkl setting
\\
\normalfont{\scriptsize{T\MakeLowercase{he present version of the paper extends results of the previous submission and concerns} L\MakeLowercase{ittlewood}--P\MakeLowercase{aley square functions in the} D\MakeLowercase{unkl setting associated with kernels satisfying mild regularity in smoothness and decay.}}}}

\author[ J. Dziuba\'nski and A. Hejna]{Jacek Dziuba\'nski and Agnieszka Hejna}


\begin{abstract}
The aim of this paper  is to prove upper and lower $L^p$ estimates, $1<p<\infty$, for Littlewood-Paley square functions in the rational Dunkl setting.
\end{abstract}

\address{J. Dziuba\'nski and A. Hejna, Uniwersytet Wroc\l awski,
Instytut Matematyczny,
Pl. Grunwaldzki 2/4,
50-384 Wroc\l aw,
Poland}
\email{jdziuban@math.uni.wroc.pl}
\email{hejna@math.uni.wroc.pl}

\thanks{
Research supported by the National Science Centre, Poland (Narodowe Centrum Nauki), Grant 2017/25/B/ST1/00599.}

\maketitle

\section{Introduction and statements of results}

On $\mathbb R^N$ equipped with normalized root system $R$ and a multiplicity function $k\geq 0$, let $\nabla f(\mathbf x)=(\partial_1f(\mathbf x),\partial_2f(\mathbf x),...,\partial_N f(\mathbf x))$, $\nabla_k f(\mathbf  x)=(T_1 f(\mathbf x),T_2f(\mathbf x),...,T_Nf(\mathbf x))$, and $\Delta_k f(\mathbf x)=\sum_{j=1}^\infty T_j^2f(\mathbf x)$ denote the classical gradient, the Dunkl gradient, and the Dunkl Laplacian respectively, where $T_j$ are the Dunkl operators (see Section \ref{sec:preliminaries}).  For two reasonable functions $f,g$ on $\mathbb R^N$, let $f*g$ stands for the Dunkl convolution.
Let
$$ f_t(\mathbf x)=t^{-\mathbf N}f(\mathbf x/t),$$ where $\mathbf N$ is the homogeneous dimension of the system $(\mathbb R^N,R,k)$ (see Section \ref{sec:preliminaries}).

Assume that
$\phi$ and  $ \psi$ are functions defined on $\mathbb R^N$ which satisfy certain smoothness and decay conditions (see Theorem \ref{teo:main1} and Corollary \ref{coro:S_t}). Assume additionally that $\int_{\mathbb{R}^N} \psi\, dw=0$, where $dw$ is the associated measure \eqref{measure_w}. We define  the following square functions:

$$ S_{\nabla_k,\phi}f(\mathbf x)=\Big(\int_0^\infty {|t\nabla_k (\phi_t*f)(\mathbf x)|^2} \frac{dt}{t}\Big)^{1/2},$$
$$ S_{\nabla,\phi}f(\mathbf x)=\Big(\int_0^\infty  {|t\nabla (\phi_t*f)(\mathbf x)|^2} \frac{dt}{t}\Big)^{1/2},$$
$$ S_\psi f(\mathbf x)= \Big(\int_0^\infty |\psi_t*f(\mathbf x)|^2\frac{dt}{t}\Big)^{1/2},$$
$$S_{\nabla_t,\phi}f(\mathbf x)=\Big(\int_0^\infty  {\Big|t \frac{d}{dt} (\phi_t*f)(\mathbf x)\Big|^2} \frac{dt}{t}\Big)^{1/2}.  $$
For $f,g\in C^2(\mathbb R^N)$ we consider the carr\'e du champ operator
\begin{equation}\label{form_Gamma} \Gamma (f,g)=\frac{1}{2}\Big(\Delta_k (f \bar g)-f\Delta_k \bar g-\bar g\Delta_k f\Big)
\end{equation}
and the associated square function
\begin{equation}\label{g_carre} \mathfrak{g}_{\Gamma,\phi}(f)(\mathbf x)=\Big(\int_0^\infty t^2\Gamma(\phi_t*f,\phi_t*f)(\mathbf x)\frac{dt}{t}\Big)^{1/2}.
\end{equation}

Let us note that $\mathfrak g_{\Gamma,\phi}$ is well--defined, since $\Gamma (f,f)(\mathbf x)\geq 0$ (see \eqref{form_Gamma1}).

We are now in a position to state our results.
 \begin{theorem}\label{teo:main1}
 {Let $s$ be a positive integer such that $2s>\mathbf N+1$.
 Assume that $\phi$, $\psi$ are $C^{2s}(\mathbb{R}^N)$ functions (not necessary radial) such that
 \begin{equation}\label{eq:main_assum} |\partial^{\beta}\phi(\mathbf x)|+|\partial^{\beta}\psi(\mathbf x)|\leq C(1+\| \mathbf x\|)^{-M-\mathbf N}\quad \text{for } \beta\in \mathbb N_0^N, \ |\beta|\leq 2s,
 \end{equation}
 for certain $M>\lfloor \mathbf N \rfloor +1$.}
 Assume also  that $\int_{\mathbb{R}^N} \psi \, dw=0$. Then, for every $1<p<\infty$, there is a constant $C_p>0$ such that for all $f \in L^p(dw)$ we have
 \begin{equation}
     \|S_{\nabla_k,\phi}f\|_{L^p(dw)} + \|S_{\nabla,\phi}f\|_{L^p(dw)} + \|S_{\psi}f\|_{L^p(dw)}+ \| \mathfrak g_{\Gamma, \phi} (f)\|_{L^p(dw)} \leq C_p\| f\|_{L^p(dw)}.
 \end{equation}
  \end{theorem}

 In order to state  lower bounds for the square functions we need additional assumptions on the functions $\phi$ and $\psi$. Let $\mathcal F$ denote the Dunkl transform (see~\eqref{eq:transform}). We say that the Dunkl transform $\mathcal F\phi$
 is {\textit{not identically zero along any direction}} if
 \begin{equation}\label{eq:non-degenarate}
\sup_{t>0} |\mathcal F\phi(t \xi)|>0 \quad \text{ for every  vector } \xi\in \mathbb R^N, \ \xi\ne 0.
\end{equation}
 This happens if e.g. $ \int_{\mathbb{R}^N} \phi \, dw\ne 0$.

 \begin{theorem}\label{teo:main2} {Assume that $\phi,\psi$ satisfy the assumptions of Theorem \ref{teo:main1} and the functions  $\mathcal F\phi$ and $\mathcal F\psi$ are not identically zero along any direction.} Then for every $1<p<\infty$ there is a constant $C_p>0$ such that for all $f \in L^p(dw)$ we have
 \begin{equation}\label{eq:lower_phi}
     \| f\|_{L^p(dw)}\leq C_p \| S_{\nabla_k, \phi}f\|_{L^p(dw)},
 \end{equation}
 \begin{equation}\label{eq:lower_psi}
     \| f\|_{L^p(dw)}\leq C_p \| S_{ \psi}f\|_{L^p(dw)},
 \end{equation}
  \begin{equation}\label{eq:lower_Gamma}
     \| f\|_{L^p(dw)}\leq C_p \| \mathfrak g_{\Gamma \phi}f\|_{L^p(dw)}.
 \end{equation}
\end{theorem}

\begin{corollary}\label{coro:S_t} {Let $s$ be a positive integer such that $2s>\mathbf N+1$. Assume that $\phi\in C^{2s+1}(\mathbb R^N)$  satisfies
\begin{equation}
    |\partial^\beta \phi(\mathbf x)|\leq C(1+\| \mathbf x\|)^{-\mathbf N-M-1} \quad \text{for } |\beta|\leq 2s+1,
\end{equation}
for certain $M>\lfloor \mathbf N \rfloor +1. $}
 Then, for every $1<p<\infty$, there is a constant $C_p>0$ such that for all $f \in L^p(dw)$ we have
 \begin{equation}\label{eq:S_t_upper}
     \| S_{\nabla_t,\phi} f\|_{L^p(dw)}\leq C_p \| f\|_{L^p(dw)}.
 \end{equation}
 If additionally the function $\mathcal F\phi$ is  not identically zero along any direction, then for every $1<p<\infty$ there is a constant $\widetilde{C}_{p}>0$ such that for all $f \in L^p(dw)$ we have
  \begin{equation}\label{eq:S_t_lower}
     \| f\|_{L^p(dw)}\leq \widetilde{C}_p \| S_{\nabla_t, \phi}f\|_{L^p(dw)}.
 \end{equation}
 \end{corollary}

 In dimension 1 and $\phi(x)=c'_k(1+x^2)^{-(\mathbf N+1)/2}$, which corresponds to the Poisson semigroup $\exp(-t\sqrt{-\Delta_k})$, $L^p$--bounds of the Littlewood-Paley square functions $S_{\nabla_k,\phi}$ and $S_{\nabla_t,\phi}$ were studied in \cite{SoltaniJFA},  \cite{LiaoZhangLi2017} and continued in higher dimensions for $1<p\leq 2$ in \cite{SoltaniJIPAM}, including the case of $\mathfrak g_\phi$ in~\cite{Ch}. In the case when $\phi(\mathbf x)=\exp(-\| \mathbf x\|^2)$ and  $1<p\leq 2$,  the upper and lower bounds for the square function $\mathfrak g_{\Gamma ,\phi}$ where proved in \cite{LiZhao}, while  the case of $2<p<\infty$ was only considered there for the particular root system, namely when the Coxeter group is isomorphic to $\mathbb Z_2^N$. We want to emphasize  that  our methods, which allow us to obtain  the bounds for the full range of $p$'s and not necessary radial functions $\phi$ and $\psi$, are different than those of \cite{LiZhao}. To prove Theorems \ref{teo:main1} and \ref{teo:main2}, we adapt  techniques of the Calder\'on-Zygmund analysis to the Dunkl setting (see Section \ref{sec:CZ}).  Then, thanks to that,  our proofs  are reduced to obtaining $L^2(dw)$ bounds and verifying that the associated kernels to the square functions  satisfy relevant estimates.

 \section{Preliminaries and notation}\label{sec:preliminaries}

The Dunkl theory is a generalization of the Euclidean Fourier analysis. It started with the seminal article \cite{Dunkl} and developed extensively afterwards (see e.g. \cite{RoeslerDeJeu}, \cite{Dunkl0}, \cite{Dunkl3}, \cite{Dunkl2},  \cite{Roesler2}, \cite{Roesle99}, \cite{Roesler2003}, \cite{ThangaveluXu}).
In this section we present basic facts concerning the theory of the Dunkl operators.  For details we refer the reader to~\cite{Dunkl},~\cite{Roesler3}, and~\cite{Roesler-Voit}.

We consider the Euclidean space $\mathbb R^N$ with the scalar product $\langle\mathbf x,\mathbf y\rangle=\sum_{j=1}^N x_jy_j
$, $\mathbf x=(x_1,...,x_N)$, $\mathbf y=(y_1,...,y_N)$, and the norm $\| \mathbf x\|^2=\langle \mathbf x,\mathbf x\rangle$. For a nonzero vector $\alpha\in\mathbb R^N$,  the reflection $\sigma_\alpha$ with respect to the hyperplane $\alpha^\perp$ orthogonal to $\alpha$ is given by
\begin{equation}\label{eq:refl}
\sigma_\alpha (\mathbf x)=\mathbf x-2\frac{\langle \mathbf x,\alpha\rangle}{\| \alpha\| ^2}\alpha.
\end{equation}
In this paper we fix a normalized root system in $\mathbb R^N$, that is, a finite set  $R\subset \mathbb R^N\setminus\{0\}$ such that   $\sigma_\alpha (R)=R$ and $\|\alpha\|=\sqrt{2}$ for every $\alpha\in R$. The finite group $G$ generated by the reflections $\sigma_\alpha \in R$ is called the {\it Weyl group} ({\it reflection group}) of the root system. A~{\textit{multiplicity function}} is a $G$-invariant function $k:R\to\mathbb C$ which will be fixed and $\geq 0$  throughout this paper.
 Let
\begin{align}\label{measure_w}
dw(\mathbf x)=\prod_{\alpha\in R}|\langle \mathbf x,\alpha\rangle|^{k(\alpha)}\, d\mathbf x
\end{align}
be  the associated measure in $\mathbb R^N$, where, here and subsequently, $d\mathbf x$ stands for the Lebesgue measure in $\mathbb R^N$.
We denote by $\mathbf N=N+\sum_{\alpha \in R} k(\alpha)$ the homogeneous dimension of the system. Clearly,
\begin{align*} w(B(t\mathbf x, tr))=t^{\mathbf N}w(B(\mathbf x,r)) \ \ \text{\rm for all } \mathbf x\in\mathbb R^N, \ t,r>0
\end{align*}
and
\begin{equation}\label{eq:integral_scaled}
\int_{\mathbb R^N} f(\mathbf x)\, dw(\mathbf x)=\int_{\mathbb R^N} t^{-\mathbf N} f(\mathbf x\slash t)\, dw(\mathbf x)\ \ \text{for} \ f\in L^1(dw)  \   \text{\rm and} \  t>0.
\end{equation}
Observe that (\footnote{The symbol $\sim$ between two positive expressions means that their ratio remains between two positive constants.})
\begin{equation}\label{eq:asymp}
w(B(\mathbf x,r))\sim r^{N}\prod_{\alpha \in R} (|\langle \mathbf x,\alpha\rangle |+r)^{k(\alpha)},
\end{equation}
so $dw(\mathbf x)$ is doubling, that is, there is a constant $C>0$ such that
\begin{equation}\label{eq:doubling} w(B(\mathbf x,2r))\leq C w(B(\mathbf x,r)) \ \ \text{ for all } \mathbf x\in\mathbb R^N, \ r>0.
\end{equation}
Moreover, there exists a constant $C\ge1$ such that,
for every $\mathbf{x}\in\mathbb{R}^N$ and for every $r_2\ge r_1>0$,
\begin{equation}\label{eq:growth}
C^{-1}\Big(\frac{r_2}{r_1}\Big)^{N}\leq\frac{{w}(B(\mathbf{x},r_2))}{{w}(B(\mathbf{x},r_1))}\leq C \Big(\frac{r_2}{r_1}\Big)^{\mathbf{N}}.
\end{equation}

For $\xi \in \mathbb{R}^N$, the {\it Dunkl operators} $T_\xi$  are the following $k$-deformations of the directional derivatives $\partial_\xi$ by a  difference operator:
\begin{equation}\label{eq:Dunkl_op}
     T_\xi f(\mathbf x)= \partial_\xi f(\mathbf x) + \sum_{\alpha\in R} \frac{k(\alpha)}{2}\langle\alpha ,\xi\rangle\frac{f(\mathbf x)-f(\sigma_\alpha (\mathbf x))}{\langle \alpha,\mathbf x\rangle}.
\end{equation}
The Dunkl operators $T_{\xi}$, which were introduced in~\cite{Dunkl}, commute and are skew-symmetric with respect to the $G$-invariant measure $dw$.
For two reasonable functions $f,g$ we have the following integration by parts formula
\begin{equation}\label{eq:by_parts}
    \int_{\mathbb{R}^N}T_{\xi}f(\mathbf{x})g(\mathbf{x})\,dw(\mathbf{x})=-\int_{\mathbb{R}^N}f(\mathbf{x})T_{\xi}g(\mathbf{x})\,dw(\mathbf{x}).
\end{equation}

For fixed $\mathbf y\in\mathbb R^N$ the {\it Dunkl kernel} $E(\mathbf x,\mathbf y)$ is the unique analytic solution to the system
\begin{equation}\label{eq:Dunkl_kernel_definition}
    T_\xi f=\langle \xi,\mathbf y\rangle f, \ \ f(0)=1.
\end{equation}
The function $E(\mathbf x ,\mathbf y)$, which generalizes the exponential  function $e^{\langle \mathbf x,\mathbf y\rangle}$, has the unique extension to a holomorphic function on $\mathbb C^N\times \mathbb C^N$. Moreover, it satisfies
{\begin{equation}\label{eq:E_prop}
    E(\mathbf{x},\mathbf{y})=E(\mathbf{y},\mathbf{x}) \text{ and }E(\lambda\mathbf{x},\mathbf{y})=E(\mathbf{x},\lambda\mathbf{y})
\end{equation}
for all $\mathbf{x},\mathbf{y} \in \mathbb{C}^N$ and $\lambda \in \mathbb{C}$.} The following theorem was proved in~\cite{Roesle99}.
\begin{theorem}[{\cite[Corollary 5.4]{Roesle99}}]\label{teo:Roesler_Dunkl_kernel}
For all $\mathbf{x}, \mathbf{z} \in \mathbb{R}^N$ and $\nu \in \mathbb{N}_0^{N}$ we have
\begin{align*}
    |\partial^{\nu}_{\mathbf{z}}E(\mathbf{x},i\mathbf{z})| \leq \|\mathbf{x}\|^{|\nu|}.
\end{align*}
\end{theorem}

Let $\{e_j\}_{1 \leq j \leq N}$ denote the canonical orthonormal basis in $\mathbb R^N$ and let $T_j=T_{e_j}$. For multi-index $\beta=(\beta_1,\beta_2,\ldots,\beta_N)  \in\mathbb N_0^N$, we set
$$ |\beta|=\beta_1+\beta_2 +\ldots +\beta_N,$$
$${\partial_{j}^{0}=I},\;\partial^{\beta}=\partial_1^{\beta_1} \circ \partial_2^{\beta_2}\circ \ldots \circ \partial_N^{\beta_N},$$
$${T_{j}^{0}=I},\;T^{\beta}=T_1^{\beta_1} \circ T_2^{\beta_2}\circ \ldots \circ T_N^{\beta_N}.$$
The Dunkl transform
  \begin{equation}\label{eq:transform}\mathcal F f(\xi)=c_k^{-1}\int_{\mathbb R^N} E(-i\xi, \mathbf x)f(\mathbf x)\, dw(\mathbf x),
  \end{equation}
  where
  $$c_k=\int_{\mathbb{R}^N}e^{-\frac{\|\mathbf{x}\|^2}{2}}\,dw(\mathbf{x})>0,$$
   originally defined for $f\in L^1(dw)$, is an isometry on $L^2(dw)$, i.e.,
   \begin{equation}\label{eq:Plancherel}
       \|f\|_{L^2(dw)}=\|\mathcal{F}f\|_{L^2(dw)} \text{ for all }f \in L^2(dw),
   \end{equation}
and preserves the Schwartz class of functions $\mathcal S(\mathbb R^N)$ (see \cite{deJeu}). Its inverse $\mathcal F^{-1}$ has the form
  \begin{align*} \mathcal F^{-1} g(\mathbf{x})=c_k^{-1}\int_{\mathbb R^N} E(i\xi, \mathbf x)g(\xi)\, dw(\xi).
  \end{align*}
Moreover,
\begin{equation}\label{eq:der_transform}
    \mathcal{F}(T_j f)(\xi)=i\xi_j\mathcal{F}f(\xi).
\end{equation}
The {\it Dunkl translation\/} $\tau_{\mathbf{x}}f$ of a function $f\in\mathcal{S}(\mathbb{R}^N)$ by $\mathbf{x}\in\mathbb{R}^N$ is defined by
\begin{align*}
\tau_{\mathbf{x}} f(\mathbf{y})=c_k^{-1} \int_{\mathbb{R}^N}{E}(i\xi,\mathbf{x})\,{E}(i\xi,\mathbf{y})\,\mathcal{F}f(\xi)\,{dw}(\xi).
\end{align*}
  It is a contraction on $L^2(dw)$, however it is an open  problem  if the Dunkl translations are bounded operators on $L^p(dw)$ for $p\ne 2$.

  {The \textit{Dunkl convolution\/} $f*g$ of two reasonable functions (for instance Schwartz functions) is defined by
  \begin{equation}\label{eq:conv_def}
      (f*g)(\mathbf{x})=c_k\,\mathcal{F}^{-1}[(\mathcal{F}f)(\mathcal{F}g)](\mathbf{x})=\int_{\mathbb{R}^N}(\mathcal{F}f)(\xi)\,(\mathcal{F}g)(\xi)\,E(\mathbf{x},i\xi)\,dw(\xi) \text{ for }\mathbf{x}\in\mathbb{R}^N,
  \end{equation}
or, equivalently, by}
\begin{align*}
  {(f{*}g)(\mathbf{x})=\int_{\mathbb{R}^N}f(\mathbf{y})\,\tau_{\mathbf{x}}g(-\mathbf{y})\,{dw}(\mathbf{y})=\int_{\mathbb R^N} f(\mathbf y)g(\mathbf x,\mathbf y) \,dw(\mathbf{y}) \text{ for all } \mathbf{x}\in\mathbb{R}^N},
\end{align*}
where, here and subsequently,
\begin{equation}\label{eq:translation} g(\mathbf x,\mathbf y)=\tau_{\mathbf x}g(-\mathbf y) =\tau_{-\mathbf y}g(\mathbf x)=g(-\mathbf y,-\mathbf x).
\end{equation} Let us point out that it is not known if the Young inequality $\| f*g\|_{L^p(dw)}\leq C\| f\|_{L^1(dw)}\| g\|_{L^p(dw)}$ holds in the Dunkl setting, unless $p=2$ or $G=\mathbb Z_2^N$ or one of the functions is radial.

The {\it Dunkl Laplacian} associated with $R$ and $k$  is the differential-difference operator $\Delta_k=\sum_{j=1}^N T_{j}^2$, which  acts on $C^2(\mathbb{R}^N)$-functions by

\begin{align*}
    \Delta_k f(\mathbf x)=\Delta_{\rm eucl} f(\mathbf x)+\sum_{\alpha\in R} k(\alpha) \delta_\alpha f(\mathbf x),
\end{align*}
\begin{align*}
    \delta_\alpha f(\mathbf x)=\frac{\partial_\alpha f(\mathbf x)}{\langle \alpha , \mathbf x\rangle} - \frac{\|\alpha\|^2}{2} \frac{f(\mathbf x)-f(\sigma_\alpha \mathbf x)}{\langle \alpha, \mathbf x\rangle^2}.
\end{align*}
Obviously, $\mathcal F(\Delta_k f)(\xi)=-\| \xi\|^2\mathcal Ff(\xi)$. The operator $\Delta_k$ is essentially self-adjoint on $L^2(dw)$ (see for instance \cite[Theorem\;3.1]{AH}) and generates the semigroup $e^{t\Delta_k}$  of linear self-adjoint contractions on $L^2(dw)$. The semigroup has the form
  \begin{align*}
  e^{t\Delta_k} f(\mathbf x)=\mathcal F^{-1}(e^{-t\|\xi\|^2}\mathcal Ff(\xi))(\mathbf x)=\int_{\mathbb R^N} h_t(\mathbf x,\mathbf y)f(\mathbf y)\, dw(\mathbf y),
  \end{align*}
  where the heat kernel
  \begin{equation}\label{eq:heat_def}
      h_t(\mathbf x,\mathbf y)=\tau_{\mathbf x}h_t(-\mathbf y),
      \end{equation}
      \begin{equation}\label{eq:heat_h_t} h_t(\mathbf x)=\mathcal F^{-1} (e^{-t\|\xi\|^2})(\mathbf x)=c_k^{-1} (2t)^{-\mathbf N\slash 2}e^{-\| \mathbf x\|^2\slash (4t)}
  \end{equation}
  is a $C^\infty$-function of all variables $\mathbf x,\mathbf y \in \mathbb{R}^N$, $t>0$, and satisfies \begin{align*} 0<h_t(\mathbf x,\mathbf y)=h_t(\mathbf y,\mathbf x),
  \end{align*}
 \begin{align*} \int_{\mathbb R^N} h_t(\mathbf x,\mathbf y)\, dw(\mathbf y)=1.
 \end{align*}
  Set
\begin{equation}\label{eq:d}
d(\mathbf x,\mathbf y)=\min_{\sigma\in G} \| \mathbf x-\sigma(\mathbf y)\|,
\end{equation}

$$V(\mathbf x,\mathbf y,t)=\max (w(B(\mathbf x,t)),w(B(\mathbf y, t))),\quad \mathcal V(\mathbf x,\mathbf y,t)=w(B(\mathbf x,t))^{1/2}w(B(\mathbf y,t))^{1/2}.$$
Note that by~\eqref{eq:asymp} and~\eqref{eq:doubling} we have
\begin{equation}\label{eq:ball_ball}
    V(\mathbf x,\mathbf y,d(\mathbf x,\mathbf y))\sim\mathcal V(\mathbf x,\mathbf y,d(\mathbf x,\mathbf y))\sim w(B(\mathbf x,d(\mathbf x,\mathbf y)))\sim w(B(\mathbf y,d(\mathbf x,\mathbf y))).
    \end{equation}
The following theorem was proved in~\cite[Theorem 4.1]{ADzH}.

\begin{theorem}\label{teo:heat}
{{\rm(a) Time derivatives\,:}
for any non-negative integer $m$,} there are constants \,$C,c>0$ such that
\begin{equation}\label{Gauss}
\left|{\partial_t^m}\hspace{.5mm}h_t(\mathbf{x},\mathbf{y})\right|\leq C\,{t^{-m}}\,V(\mathbf{x},\mathbf{y},\!\sqrt{t\,})^{-1}\,e^{-\hspace{.25mm}c\hspace{.5mm}d(\mathbf{x},\mathbf{y})^2\slash t},
\end{equation}
{for every \,$t\hspace{-.5mm}>\hspace{-.5mm}0$ and for every \,$\mathbf{x},\mathbf{y}\!\in\hspace{-.5mm}\mathbb{R}^N$.}
\par\noindent
{{\rm(b) H\"older  bounds\,:}
for any nonnegative integer $m$,} there are constants \,$C,c>0$ such that
\begin{equation}\label{Holder}
\left|{\partial_t^{{m}}}h_t(\mathbf{x},\mathbf{y})-{\partial_t^{{m}}}h_t(\mathbf{x},\mathbf{y}')\right|\leq C\,t^{-{m}}\,\Bigl(\frac{{\|}\mathbf{y}\!-\!\mathbf{y}'{\|}}{\sqrt{t\,}}\Bigr)\,V(\mathbf{x},\mathbf{y},\!\sqrt{t\,})^{-1}\,e^{-\hspace{.25mm}c\hspace{.5mm}d(\mathbf{x},\mathbf{y})^2\slash t},
\end{equation}
{for every \,$t\hspace{-.5mm}>\hspace{-.5mm}0$ and for every \,$\mathbf{x},\mathbf{y},\mathbf{y}'\hspace{-1mm}\in\hspace{-.5mm}\mathbb{R}^N$ such that \,${\|}\mathbf{y}\!-\!\mathbf{y}'{\|}\!<\!\sqrt{t\,}$}.
\par\noindent
{{\rm(c) Dunkl derivative\,:}
for any \,$\xi\hspace{-.5mm}\in\hspace{-.5mm}\mathbb{R}^N$ and for any nonnegative integer $m$,} there are constants \,$C,c\hspace{-.5mm}>\hspace{-.5mm}0$ such that
\begin{equation}\label{TxiDtHeat}
\Bigl|\hspace{.5mm}T_{{\xi},\mathbf{x}}\,{\partial_t^m}\hspace{.5mm}h_t(\mathbf{x},\mathbf{y})\Bigr|\leq C\,t^{-m-1\slash 2}\,V(\mathbf{x},\mathbf{y},\!\sqrt{t\,})^{-1}\,e^{-\hspace{.25mm}c\hspace{.5mm}d(\mathbf{x},\mathbf{y})^2\slash t}\,{,}
\end{equation}
{for all \,$t\hspace{-.5mm}>\hspace{-.5mm}0$ and \,$\mathbf{x},\mathbf{y}\!\in\hspace{-.5mm}\mathbb{R}^N$}.
\par\noindent
{{\rm(d) Mixed derivatives\,:}
for any nonnegative integer \,$m$ and for any multi-indices \,$\alpha,\beta$, there are constants \,$C,c\hspace{-.5mm}>\hspace{-.5mm}0$ such that, for every \,$t>0$ and for every \,$\mathbf{x},\mathbf{y}\in\mathbb{R}^N$,
\begin{equation}\label{DtDxDyHeat}
\bigl|\hspace{.25mm}\partial_t^m\partial_{\mathbf{x}}^{\alpha}\partial_{\mathbf{y}}^{\beta}h_t(\mathbf{x},\mathbf{y})\bigr|\le C\,t^{-m-\frac{|\alpha|}2-\frac{|\beta|}2}\,V(\mathbf{x},\mathbf{y},\!\sqrt{t\,})^{-1}\,e^{-\hspace{.25mm}c\hspace{.5mm}d(\mathbf{x},\mathbf{y})^2\slash t},
\end{equation}
for every \,$t\hspace{-.5mm}>\hspace{-.5mm}0$ and for every \,$\mathbf{x},\mathbf{y}\!\in\hspace{-.5mm}\mathbb{R}^N$.}
\end{theorem}

 We finish this section by the proposition which will be used in proving the bounds of the square functions.
\begin{proposition}\label{prop:ord_implies_Dunkl}
Let $M \geq 0$ and let $\ell$ be a positive integer. Assume that $\phi \in C^{\ell}(\mathbb{R}^N)$ satisfies
\begin{equation}\label{eq:induction_ord}
    |\partial^{\beta}\phi(\mathbf{x})| \leq (1+\|\mathbf{x}\|)^{-\mathbf{N}-M} \text{ for all }|\beta| \leq \ell.
\end{equation}
There is a constant $C>0$ such that
\begin{equation}\label{eq:induction_dunkl}
    |T^{\beta}\phi(\mathbf{x})| \leq C(1+\|\mathbf{x}\|)^{-\mathbf{N}-M} \text{ for all }|\beta| \leq \ell.
\end{equation}
Moreover, if additionally  $M>0$ and $m$ is a positive integer such that $m<M$, then $\mathcal F\phi\in C^m(\mathbb R^N)$ and
\begin{equation}\label{transform_bound}
    |\partial^{\beta'}_\xi \mathcal F\phi(\xi)|\leq C_{\beta'} (1+\|\xi\|)^{-\ell} \quad \text{for all } |\beta'|\leq  m.
\end{equation}
\end{proposition}

\begin{proof}
The proof {of~\eqref{eq:induction_dunkl}} is by induction on $\ell$. Assume that~\eqref{eq:induction_ord} implies~\eqref{eq:induction_dunkl} for $\ell_1$. We will prove that~\eqref{eq:induction_ord} implies~\eqref{eq:induction_dunkl} for $\ell_1+1$. Let $\alpha \in R$. By the definition of $T_{j}$, it is enough to show that there is a constant $C>0$ such that the function
$$\mathbf x\longmapsto C^{-1}\frac{\phi(\mathbf{x})-\phi(\sigma_{\alpha}(\mathbf{x}))}{\langle \mathbf{x},\alpha \rangle}$$ satisfies~\eqref{eq:induction_ord} with $\ell=\ell_1$. Let $\beta \in \mathbb{N}_0^{N}$ be such that $|\beta| \leq \ell_1$. We consider two cases.
\\
\textbf{Case 1.} $|\langle \mathbf{x},\alpha \rangle| \geq {1/10}$. Note that there is a constant $C=C_{\ell_1}>0$ independent of $\mathbf{x}$ such that for all $|\beta'| \leq \ell_1$ we have
\begin{equation}\label{eq:scalar_est}
    |\partial^{\beta'}_{\mathbf{x}}(\langle \mathbf{x},\alpha \rangle^{-1})| \leq C.
\end{equation}
Thanks to the Leibniz rule, the estimate for $\partial^{\beta}\Big(\frac{\phi(\mathbf{x})-\phi(\sigma_{\alpha}(\mathbf{x}))}{\langle \mathbf{x},\alpha \rangle}\Big)$ is a consequence of~\eqref{eq:induction_ord} and~\eqref{eq:scalar_est}.
\\
\textbf{Case 2.} $|\langle \mathbf{x},\alpha \rangle| < {1/10}$. We have
\begin{equation}\label{eq:diff_argument}
\begin{split}
    \frac{\phi(\mathbf{x})-\phi(\sigma_{\alpha}(\mathbf{x}))}{\langle \mathbf{x}, \alpha \rangle}&=\langle \mathbf{x}, \alpha \rangle^{-1}\int_0^{1} \frac{d}{dt}(\phi(\mathbf{x}-2t\alpha \|\alpha\|^{-2}\langle \mathbf{x}, \alpha\rangle ))\,dt\\&=c_{\alpha} \int_0^1 \langle \nabla_{\mathbf{x}} \phi(\mathbf{x}-2t\alpha \|\alpha\|^{-2}\langle \mathbf{x}, \alpha\rangle ), \alpha \rangle \,dt.
\end{split}
\end{equation}
By the assumption $|\langle \mathbf{x},\alpha \rangle| < {1/10}$, so for all $t \in [0,1]$ we have
\begin{equation}\label{eq:diff_sim}
    \|\mathbf{x}-2t\alpha \|\alpha\|^{-2}\langle \mathbf{x}, \alpha\rangle\| \sim \|\mathbf{x}\|,
\end{equation}
so, by~\eqref{eq:induction_ord} with $\ell=\ell_1+1$ we obtain
\begin{align*}
    &\left|\partial^{\beta}_{\mathbf{x}}\Big(\frac{\phi(\mathbf{x})-\phi(\sigma_{\alpha}(\mathbf{x}))}{\langle \mathbf{x}, \alpha \rangle}\Big)\right| \leq \left|c_{\alpha} \int_0^1 \langle \partial^{\beta}_{\mathbf{x}}\nabla_{\mathbf{x}} \phi(\mathbf{x}-2t\alpha \|\alpha\|^{-2}\langle \mathbf{x}, \alpha\rangle ), \alpha \rangle \,dt\right| \\&\leq C\int_{0}^{1}\left(1+\|\mathbf{x}-2t\alpha \|\alpha\|^{-2}\langle \mathbf{x}, \alpha\rangle\|\right)^{-\mathbf{N}-M}\,dt \leq C(1+\|\mathbf{x}\|)^{-\mathbf{N}-M},
\end{align*}
{which completes the proof of \eqref{eq:induction_dunkl}.

We now turn to prove \eqref{transform_bound}. By {~\eqref{eq:E_prop} and Theorem~\ref{teo:Roesler_Dunkl_kernel} we have $  |\partial^\beta_{\xi} E(-i\mathbf{x},\xi)|\leq C\|\mathbf{x}\|^{|\beta|}$.}
   Thus, using~\eqref{eq:der_transform}, we obtain
    \begin{equation}\label{eq:7A}
        \begin{split}
      \Big|   \partial_\xi^{\beta'} \Big(\xi^\beta \mathcal F\phi (\xi)\Big) \Big| &=c_k^{-1} \Big| \int_{\mathbb{R}^N} (T^\beta \phi)(\mathbf{x})\partial_\xi^{\beta'} E(-i\mathbf{x},\xi )\, dw(\mathbf{x})\Big|\\
      &\leq C \int_{\mathbb{R}^N} (1+\|\mathbf{x}\|)^{-M-\mathbf N}(1+\|\mathbf{x}\|)^{m}\, dw(\mathbf{x})
      \leq C_{\beta',\beta}
        \end{split}
    \end{equation}
    for $|\beta'|\leq m$ and {$|\beta |\leq \ell$.}
From \eqref{eq:7A} we conclude that for $|\beta'|\leq m$ and {$|\beta |\leq \ell$} one has
\begin{equation}\label{eq:7B}
        \begin{split}
      \Big| \xi^\beta  \partial_\xi^{\beta'}  \mathcal F\phi (\xi) \Big| &
      \leq C_{\beta',\beta}'.
        \end{split}
    \end{equation}
    Now \eqref{transform_bound} can be deduced easily from \eqref{eq:7B}. }
\end{proof}

\section{Vector valued Calder\'on-Zygmund analysis in the Dunkl setting}\label{sec:CZ} The proof of main results will be based on the following straightforward adaptation of the vector valued approach to square functions in the Dunkl setting (cf.~\cite{Duo},~\cite{St1}). Since the conditions on  kernels are expressed by means of both: the Euclidean distance $\| \mathbf x-\mathbf y\|$ and the distance $d(\mathbf x,\mathbf y)$, for the convenience of the reader we present the details.

\subsection{Vector valued Calder\'on-Zygmund operators in the Dunkl setting.}
Let $\mathcal H_1$ and $\mathcal H_2$ be separable Hilbert spaces. We shall consider the vector valued $L^p(dw,\mathcal H_j)$ spaces with the norms
$$ \| f\|^p_{L^p(dw, \mathcal H_j)}=\int_{\mathbb{R}^N} \| f(\mathbf x)\|_{\mathcal H_j}^p\, dw(\mathbf x).$$
Note that $L^2(dw, \mathcal H_j)=:\mathbf H_j$ is a Hilbert space with the inner product
$$ \langle f,g\rangle_{\mathbf H_j}=\int_{\mathbb{R}^N}\langle f(\mathbf x), g(\mathbf x)\rangle_{\mathcal H_j}\, dw(\mathbf x).$$

\begin{proposition}\label{Calderon-vector}
Let $\mathcal K$ be a bounded linear operator from $L^2(dw,\mathcal H_1)$ into $L^2(dw,\mathcal H_2)$ with an associated operator valued  kernel $\mathcal K(\mathbf x,\mathbf y)\in\mathcal L(\mathcal H_1,\mathcal H_2)$ for $d(\mathbf x,\mathbf y)>0$. Assume that there are constants $C,\delta'>0$ such that  for all $\mathbf{x},\mathbf{y},\mathbf{y}' \in \mathbb{R}^N$ such that $2\| \mathbf y-\mathbf y'\|\leq d(\mathbf x,\mathbf y)$ one has
\begin{equation}\label{eq:kernel_assumption_1}
    \| \mathcal K(\mathbf x,\mathbf y)-\mathcal K(\mathbf x,\mathbf y')\|_{\mathcal L(\mathcal H_1,\mathcal H_2)}\leq \frac{C}{V(\mathbf x,\mathbf y,d(\mathbf x,\mathbf y))}\Big(\frac{\|\mathbf y-\mathbf y'\| }{d(\mathbf x,\mathbf y)}\Big)^{\delta'},
    \end{equation}
\begin{equation}\label{eq:kernel_assumption_2}
    \| \mathcal K(\mathbf y,\mathbf x)-\mathcal K(\mathbf y',\mathbf x)\|_{\mathcal L(\mathcal H_1,\mathcal H_2)}\leq \frac{C}{V(\mathbf x,\mathbf y,d(\mathbf x,\mathbf y))}\Big(\frac{\|\mathbf y-\mathbf y'\| }{d(\mathbf x,\mathbf y)}\Big)^{\delta'}.
\end{equation}
Then, for every $1<p<\infty$, the operator $\mathcal K$, initially defined on $L^p(dw,\mathcal H_1)\cap L^2(dw,\mathcal H_1)$, has a unique extension to a bounded operator from  $L^p(dw,\mathcal H_1)$ to $L^p(dw,\mathcal H_2)$.
\end{proposition}

\begin{proof}
First we prove the weak type $(1,1)$ estimate. Consider $f\in L^1(dw,\mathcal H_1)$. Fix $\lambda>0$. We denote by $\mathcal Q_\lambda$ the collection of all maximal (disjoint)  dyadic cubes $Q_j$ in $\mathbb R^N$ satisfying
 \begin{equation}\label{eq:CZ_decomp0} \lambda< \frac{1}{w(Q_j)}\int_{Q_j}\|f(\mathbf{x})\|_{\mathcal{H}_1}\,dw(\mathbf{x}).
 \end{equation}
 Then, thanks to~\eqref{eq:doubling}, we have
 \begin{equation}\label{eq:CZ_decom_1}
     \frac{1}{w(Q_j)} \int_{Q_j}\|f(\mathbf{x})\|_{\mathcal{H}_1}\,dw(\mathbf{x})\leq C_1\lambda.
 \end{equation}
Let
\begin{equation*}\begin{split}
& f(\mathbf x)=\mathbf{g}(\mathbf x)+\sum_{Q_{j} \in \mathcal{Q}_{\lambda}} (f(\mathbf x)-f_{Q_j})\chi_{Q_j}(\mathbf x)=\mathbf{g}(\mathbf x)+\sum_{Q_{j} \in \mathcal{Q}_{\lambda}} \mathbf{b}_j(\mathbf{x})= \mathbf{g}(\mathbf x)+\mathbf{b}(\mathbf x),\\
&f_{Q_j}=\frac{1}{w(Q_j)}\int_{Q_j} f(\mathbf x)\, dw(\mathbf x),
\end{split}
\end{equation*}
be the Calder\'on-Zygmund decomposition at the level $\lambda$.
Let $Q_j^{*}$ be the cube which has the same center $\mathbf y_j$ as $Q_j$ but whose diameter is expanded by the factor 2. Set $\Omega =\bigcup\limits_{Q_{j} \in \mathcal{Q}_{\lambda}} \mathcal O(Q_j^*)$, where, for a Lebesgue measurable set $U\subset\mathbb R^N$, we denote $\mathcal O(U)=\{\sigma (\mathbf x): \mathbf x\in U,\ \sigma\in G\}$. Then by~\eqref{eq:doubling} and~\eqref{eq:CZ_decomp0}, we have
$$ w(\Omega )\leq \sum_{Q_{j} \in \mathcal{Q}_{\lambda}} w(\mathcal O(Q_j^*))\leq C\sum_{Q_{j} \in \mathcal{Q}_{\lambda}} w(Q_j)\leq  C \lambda^{-1}\| f\|_{L^1(dw,\mathcal H_1)}.
$$
Clearly,
$$ \{\mathbf x\in \Omega^c: \| \mathcal Kf(\mathbf x)\|_{\mathcal H_2}>\lambda\}\subseteq \{\mathbf x\in \Omega^c: \| \mathcal K\mathbf{g}(\mathbf x)\|_{\mathcal H_2}>\lambda/2 \}\cup \{\mathbf x\in \Omega^c: \| \mathcal K\mathbf{b}(\mathbf x)\|_{\mathcal H_2}>\lambda/2\}.$$
By the boundendess of $\mathcal K$ from $L^2(dw,\mathcal H_1)$ to $L^2(dw,\mathcal H_2)$, we get
\begin{equation}\label{eq:good_function}
    w( \{\mathbf x\in \Omega^c: \| \mathcal K\mathbf{g}(\mathbf x)\|_{\mathcal H_2}>\lambda/2 \} )\leq C \frac{\| \mathbf{g}\|_{L^2(dw,\mathcal H_1)}^2}{ \lambda^2}\leq C \frac{\lambda   \| f\|_{L^1(dw,\mathcal H_1)}}{\lambda^2}.
\end{equation}
In order to estimate the  measure of the second term we recall that  $\int_{\mathbb{R}^N}\mathbf b_j\, dw=0$  and  write
\begin{equation}\label{eq:CZ_decomp_split}
\begin{split}
   &w(\{\mathbf x\in \Omega^c: \| \mathcal K\mathbf{b}(\mathbf x)\|_{\mathcal H_2}>\lambda/2\}) \leq C\lambda^{-1}\|\chi_{\Omega^{c}}\mathcal{K}\mathbf{b}\|_{L^1(dw,\mathcal{H}_2)}\\
   &=C\lambda^{-1}\int_{\Omega^{c}}\left\|\sum_{Q_{j}\in \mathcal{Q}_{\lambda}}\int_{\mathbb{R}^{N}}\mathcal{K}(\mathbf{x},\mathbf{y})\mathbf{b}_{j}(\mathbf{y})\,dw(\mathbf{y})\right\|_{\mathcal{H}_2}\,dw(\mathbf{x})\\
   &=C\int_{\Omega^{c}}\left\|\sum_{Q_{j}\in \mathcal{Q}_{\lambda}}\int_{\mathbb{R}^{N}}[\mathcal{K}(\mathbf{x},\mathbf{y})-\mathcal{K}(\mathbf{x},\mathbf{y}_{j})]\mathbf{b}_{j}(\mathbf{y})\,dw(\mathbf{y})\right\|_{\mathcal{H}_2}\,dw(\mathbf{x}) \\&\leq C\sum_{Q_{j} \in \mathcal{Q}_{\lambda}}\int_{\Omega^{c}}\int_{Q_{j}}\| \mathcal K(\mathbf x,\mathbf y)-\mathcal K(\mathbf x,\mathbf y_j)\|_{\mathcal L(\mathcal H_1,\mathcal H_2)}\|\mathbf{b}_{j}(\mathbf{y})\|_{\mathcal{H}_1}\,dw(\mathbf{y})\,dw(\mathbf{x}) \\&\leq \sum_{Q_{j} \in \mathcal{Q}_{\lambda}}\int_{Q_{j}}\|{b}_{j}(\mathbf{y})\|_{\mathcal{H}_1}\int_{(\mathcal{O}(Q_{j}^{*}))^{c}}\frac{C}{V(\mathbf x,\mathbf y,d(\mathbf x,\mathbf y))}\Big(\frac{\|\mathbf y-\mathbf y_{j}\| }{d(\mathbf x,\mathbf y)}\Big)^{\delta'}\,dw(\mathbf{x})\,dw(\mathbf{y}).
\end{split}
\end{equation}
Clearly, there is a constant $C>0$ such that for all $\mathbf{y} \in Q_{j} \in \mathcal{Q}_{\lambda}$ we have
\begin{align*}
    \int_{(\mathcal{O}(Q_{j}^{*}))^{c}}\frac{1}{V(\mathbf x,\mathbf y,d(\mathbf x,\mathbf y))}\Big(\frac{\|\mathbf y-\mathbf y'_j\|}{d(\mathbf x,\mathbf y)}\Big)^{\delta'}\,dw(\mathbf{x}) \leq C,
\end{align*}
so, by~\eqref{eq:CZ_decomp_split}, we get
\begin{equation}\label{eq:bad_function}
    w(\{\mathbf x\in \Omega^c: \| \mathcal K\mathbf{b}(\mathbf x)\|_{\mathcal H_2}>\lambda/2\}) \leq \frac{C}{\lambda}\sum_{Q_{j} \in \mathcal{Q}_{\lambda}}\int_{\mathbb{R}^{N}}\|\mathbf{b}_{j}(\mathbf{y})\|_{\mathcal{H}_1}\,dw(\mathbf{y}) \leq \frac{C}{\lambda}\|f\|_{L^1(dw,\mathcal{H}_1)}.
\end{equation}
By~\eqref{eq:good_function} and~\eqref{eq:bad_function} we obtain that $\mathcal{K}$ is of weak type $(1,1)$. Thanks to the vector-valued version of the Marcinkiewicz interpolation theorem (see e.g.~\cite[Exercise 5.5.3]{Grafakos_classical}) we obtain the claim for $1<p \leq 2$.

To prove the bounds for $2<p<\infty$, we apply the well-known  duality argument. Observe that $\mathcal K^*$ is a bounded operator from the Hilbert space $L^2(dw,\mathcal H_2)$ to the Hilbert space  $L^2(dw,\mathcal H_1)$ with the associated kernel $\mathcal K^*(\mathbf x,\mathbf y)=\mathcal K(\mathbf y,\mathbf x)^*\in\mathcal L(\mathcal H_2,\mathcal H_1)$.  Hence,
 $ \| \mathcal K^*(\mathbf x,\mathbf y)\|_{\mathcal L(\mathcal H_2,\mathcal H_1)}=  \| \mathcal K(\mathbf y,\mathbf x)\|_{\mathcal L(\mathcal H_1,\mathcal H_2)}<\infty$ for $d(\mathbf x,\mathbf y)>0$ and, if  $ {2} \| \mathbf y-\mathbf y'\|\leq d(\mathbf x,\mathbf y)$, then, by \eqref{eq:kernel_assumption_2},
\begin{equation}\begin{split}\label{eq:kernel_assumption_3}
    \| \mathcal K^*(\mathbf x,\mathbf y)-\mathcal K^*(\mathbf x,\mathbf y')\|_{\mathcal L(\mathcal H_2,\mathcal H_1)}
    &=  \| \mathcal K(\mathbf y,\mathbf x)-\mathcal K(\mathbf y',\mathbf x)\|_{\mathcal L(\mathcal H_1,\mathcal H_2)}\\
    &\leq \frac{C}{V(\mathbf x,\mathbf y,d(\mathbf x,\mathbf y))}\Big(\frac{\|\mathbf y-\mathbf y'\| }{d(\mathbf x,\mathbf y)}\Big)^{\delta'}.
\end{split}\end{equation}
Let $\frac{1}{p}+\frac{1}{p'}=1$. Note that $1<p'<2$. Consequently, from the first part of the proof we conclude that there is a constant $C_{p'}>0$ such that for all $g \in L^{p'}(dw,\mathcal{H}_2)$ we have
\begin{equation}\label{eq:conjugate} \| \mathcal K^* g\|_{L^{p'}(dw,\mathcal H_1)}\leq C_{p'}\| g\|_{L^{p'}(dw,\mathcal H_2)}.
\end{equation}
Now, for $f\in L^p(dw,\mathcal H_1)\cap  L^2(dw,\mathcal H_1)$, $2<p<\infty$, we write
\begin{equation}
    \begin{split}
        \| \mathcal Kf\|_{L^p(dw,\mathcal H_2)}&=\sup_{\stackrel{g\in L^{p'}(dw,\mathcal H_2)\cap  L^2(dw,\mathcal H_2)\; } {\|g\|_{L^{p'}(dw,\mathcal{H}_2)}=1}}\Big| \int_{\mathbb{R}^N}\langle \mathcal Kf(\mathbf x),g(\mathbf x)\rangle_{\mathcal H_2}\, dw(\mathbf x)\Big|\\
        &=\sup_{\stackrel{g\in L^{p'}(dw,\mathcal H_2)\cap  L^2(dw,\mathcal H_2)}{\|g\|_{L^{p'}(dw,\mathcal{H}_2)}=1}} \Big| \int_{\mathbb{R}^N}\langle f(\mathbf x),\mathcal K^* g(\mathbf x)\rangle_{\mathcal H_1}\, dw(\mathbf x)
        \Big|\\
        &\leq \sup_{\stackrel{g\in L^{p'}(dw,\mathcal H_2)\cap  L^2(dw,\mathcal H_2)}{\|g\|_{L^{p'}(dw,\mathcal{H}_2)}=1}} \int_{\mathbb{R}^N}\| f(\mathbf x)\|_{\mathcal H_1}\|\mathcal K^* g(\mathbf x)\|_{\mathcal H_1} \, dw(\mathbf x)\\
        &\leq \sup_{\stackrel{g\in L^{p'}(dw,\mathcal H_2)\cap  L^2(dw,\mathcal H_2)}{\|g\|_{L^{p'}(dw,\mathcal{H}_2)}=1}} \| f\|_{L^p(dw,\mathcal H_1)} \| \mathcal K^*g\|_{L^{p'}(dw,\mathcal H_1)}\\
        &\leq C_{p'}\| f\|_{L^p(dw,\mathcal H_1)},
    \end{split}
\end{equation}
 where in the last inequality we have used \eqref{eq:conjugate}.
\end{proof}
\subsection{Vector valued approach to square functions.} For further applications we shall use $\mathcal H_1=\mathbb C$ and $\mathcal H_2=L^2((0,\infty),\frac{dt}{t})$.
Let $K(t,\mathbf x,\mathbf y)$ be a measurable function on $(0,\infty)\times \mathbb R^N\times\mathbb R^N$ such that for certain $M'>\mathbf N$ one has
\begin{equation}\label{eq:estimate_K11}| K(t,\mathbf x,\mathbf y)|\leq C \mathcal V(\mathbf x,\mathbf y, t)^{-1} \Big(1+\frac{d(\mathbf x,\mathbf y)}{t}\Big)^{-M'}.
\end{equation}
We additionally assume that there are  constants $ \delta>0$ and $C>0$ such that for all $\mathbf x,\mathbf y, \mathbf y' \in \mathbb{R}^N$ and $t>0$, if $ \| \mathbf y-\mathbf y'\|\leq t$, then
\begin{equation}\label{eq:Holder_K22} |K(t,\mathbf x,\mathbf y)-K(t,\mathbf x,\mathbf y')|\leq C \Big(\frac{\| \mathbf y-\mathbf y'\|}{t}\Big)^\delta\mathcal V(\mathbf x,\mathbf y,t)^{-1} \Big(1+\frac{d(\mathbf x,\mathbf y)}{t}\Big)^{-M'-\delta},
\end{equation}
\begin{equation}\label{eq:Holder_K33} |K(t,\mathbf y,\mathbf x)-K(t,\mathbf y',\mathbf x)|\leq C  \Big(\frac{\| \mathbf y-\mathbf y'\|}{t}\Big)^\delta \mathcal V(\mathbf x,\mathbf y,t)^{-1} \Big(1+\frac{d(\mathbf x,\mathbf y)}{t}\Big)^{-M'-\delta}.
\end{equation}
Let
$$K_tf(\mathbf x)=\int_{\mathbb{R}^N}K(t,\mathbf x,\mathbf y)f(\mathbf y)\, dw(\mathbf y). $$
We define the square function $S_K$ associated with the kernel $K(t,\mathbf{x},\mathbf{y})$ by
$$ S_K(f)(\mathbf x)=\Big(\int_0^\infty |K_tf(\mathbf x)|^2\frac{dt}{t}\Big)^{1/2}.$$ \begin{theorem}\label{teo:square-vector}
Assume that an integral kernel $K(t,\mathbf x,\mathbf y)$ satisfies \eqref{eq:estimate_K11}-\eqref{eq:Holder_K33} with certain $M'>\mathbf N$ and  $\delta>0$, and the associated square function $S_K$ is bounded from $L^2(dw)$ into itself. Then for every $1<p<\infty$ there is a constant $C_p>0$ such that
\begin{equation}
    \| S_K(f)\|_{L^p(dw)}\leq C_p\| f\|_{L^p(dw)} \quad \text{\rm for all } f\in L^p(dw).
\end{equation}
\end{theorem}
\begin{proof}
In order to prove   Theorem \ref{teo:square-vector}  we  note that the boundedness  of the square function $S_K$ on $L^p(dw)$-spaces is equivalent to the boundedness of the operator $\mathcal K$ from $L^p(dw,\mathcal H_1)$ to $L^p(dw,\mathcal H_2)$, $\mathcal H_1=\mathbb C$,  $\mathcal H_2=L^2((0,\infty),\,\frac{dt}{t})$, where $\mathcal K$ is defined by
$$ \mathcal Kf(\mathbf x)(t)=\int_{\mathbb{R}^N}K_t(\mathbf x,\mathbf y)f(\mathbf y)\, dw(\mathbf y).$$
Recall that the boundedness  of $\mathcal K$ from $L^2(dw,\mathcal H_1)$ to $L^2(dw,\mathcal H_2)$ is guaranteed  by one of the assumptions. Therefore,  to finish the proof, it suffices to check that the associated kernel $\mathcal K(\mathbf x,\mathbf y)=K(t,\mathbf x,\mathbf y)$ belongs to  $\mathcal L(\mathcal H_1,\mathcal H_2)$ for $d(\mathbf x,\mathbf y)>0$ and  satisfies \eqref{eq:kernel_assumption_1} and  \eqref{eq:kernel_assumption_2} with $\delta'=\min (\delta, M'-\mathbf N)>0$, and then apply Proposition \ref{Calderon-vector}.
{To verify the first requirement we shall prove    that there is a constant $C>0$ such that for $d(\mathbf x,\mathbf y)>0$ one has
\begin{equation}\label{mathcalK_bound}\|\mathcal K(\mathbf x,\mathbf y)\|_{\mathcal L(\mathcal H_1,\mathcal H_2)}\leq Cw(B(\mathbf x,d(\mathbf x,\mathbf y)))^{-1}.
\end{equation}
By \eqref{eq:estimate_K11},
\begin{equation}\begin{split}
    \| \mathcal K(\mathbf x,\mathbf y)\|_{\mathcal L(\mathcal H_1,\mathcal H_2)}^2&
    \leq C^2 \int_0^\infty \frac{1}{\mathcal V(\mathbf x,\mathbf y,t)^2}\Big(1+\frac{d(\mathbf x,\mathbf y)}{t}\Big)^{-2M'}\frac{dt}{t}\\
    &=C^2 \int_0^\infty \frac{1}{\mathcal V(\mathbf x,\mathbf y, td(\mathbf x,\mathbf y))^2}\Big(1+\frac{1}{t}\Big)^{-2M'}\frac{dt}{t}\\
    &=\frac{C^2}{\mathcal V(\mathbf x,\mathbf y,d(\mathbf x,\mathbf y))^2} \int_0^\infty \frac{\mathcal V(\mathbf x,\mathbf y,d(\mathbf x,\mathbf y)))^2}{\mathcal V(\mathbf x,\mathbf y,td(\mathbf x,\mathbf y)))^2}\Big(1+\frac{1}{t}\Big)^{-2M'}\frac{dt}{t}
\end{split}
\end{equation}
By~\eqref{eq:ball_ball} we have $\mathcal V(\mathbf x,\mathbf y,d(\mathbf x,\mathbf y))\sim w(B(\mathbf x,d(\mathbf x,\mathbf y)))$. Applying the second  inequality of \eqref{eq:growth} for $0<t<1$ and the first  one for $t \geq 1$ we get the claim.}

We  now turn to verify \eqref{eq:kernel_assumption_1}.
 Note that $d(\mathbf x,\mathbf y){ \sim d(\mathbf x,\mathbf y')}$ for $d(\mathbf x,\mathbf y)>2\| \mathbf y-\mathbf y'\|$.  Hence, by~\eqref{eq:estimate_K11} and~\eqref{eq:Holder_K22} we have
\begin{equation*}
    \begin{split}
         \| \mathcal K&(\mathbf x,\mathbf y)-\mathcal K(\mathbf x,\mathbf y')\|^2_{\mathcal L(\mathcal H_1,\mathcal H_2)} \leq C \int_0^{\| \mathbf y-\mathbf y'\|}  \mathcal V(\mathbf x,\mathbf y, t)^{-2} \Big(1+\frac{d(\mathbf x,\mathbf y)}{t}\Big)^{-2M'} \frac{dt}{t}\\
         &+ C\int_{\|\mathbf y-\mathbf y'\|}^\infty  \mathcal V(\mathbf x,\mathbf y, t)^{-2}\Big(\frac{\|\mathbf y-\mathbf y'\|}{t}\Big)^{2\delta}\Big(1+\frac{d(\mathbf x,\mathbf y)}{t}\Big)^{-2M'-2\delta} \frac{dt}{t}\\
        & \leq C \int_0^{\| \mathbf y-\mathbf y'\|/d(\mathbf x,\mathbf y)}  \mathcal V(\mathbf x, \mathbf y, td(\mathbf x,\mathbf y))^{-2} \Big(1+\frac{1}{t}\Big)^{-2M'} \frac{dt}{t}\\
        &+ C\int_{\| \mathbf y-\mathbf y'\|/d(\mathbf x,\mathbf y)}^\infty    \Big(\frac{\| \mathbf y-\mathbf y'\|}{td(\mathbf x,\mathbf y)}\Big)^{2\delta}  \mathcal V(\mathbf x, \mathbf y,td(\mathbf x,\mathbf y))^{-2}\Big(1+\frac{1}{t}\Big)^{-2M'-2\delta} \frac{dt}{t}\\
         &\leq
         \frac{C}{\mathcal V(\mathbf x,\mathbf y, d(\mathbf x,\mathbf y))^2} \int_0^{\|\mathbf y-\mathbf y'\|/d(\mathbf x,\mathbf y)} \frac{\mathcal V(\mathbf x,d(\mathbf x,\mathbf y))^2}{\mathcal V(\mathbf x, \mathbf y,td(\mathbf x,\mathbf y)))^2}t^{2M'} \frac{dt}{t}\\
         &+
         \frac{C}{\mathcal V(\mathbf x,\mathbf y,d(\mathbf x,\mathbf y))^2} \int_0^\infty \Big( \frac{\| \mathbf y-\mathbf y'\|}{d(\mathbf x,\mathbf y)}\Big)^{2\delta}t^{-2\delta}  \frac{\mathcal V(\mathbf x,\mathbf y,d(\mathbf x,\mathbf y))^2}{\mathcal V(\mathbf x, \mathbf y,td(\mathbf x,\mathbf y))^2}\Big(1+\frac{1}{t}\Big)^{-2M'-2\delta} \frac{dt}{t}.\\
         \end{split}
         \end{equation*}
         Applying \eqref{eq:growth} together with \eqref{eq:ball_ball}, we get
         \begin{equation*}
             \begin{split}
         \| \mathcal K(\mathbf x,\mathbf y)&-\mathcal K(\mathbf x,\mathbf y')\|^2_{\mathcal L(\mathcal H_1,\mathcal H_2)}  \leq  \frac{C}{w(B(\mathbf x,d(\mathbf x,\mathbf y)))^2} \int_0^{\|\mathbf y-\mathbf y'\|/d(\mathbf x,\mathbf y)} t^{2M'-2\mathbf N} \frac{dt}{t}\\
         &+ \frac{C}{w(B(\mathbf x,d(\mathbf x,\mathbf y)))^2} \int_0^\infty \Big( \frac{\| \mathbf y-\mathbf y'\|}{d(\mathbf x,\mathbf y)}\Big)^{2\delta}t^{-2\delta}  (t^{-\mathbf N}+t^{-N})^2\Big(1+\frac{1}{t}\Big)^{-2M'-2\delta} \frac{dt}{t}\\
         &\leq \frac{C}{w(B(\mathbf x,d(\mathbf x,\mathbf y)))^2} \Big( \frac{\|\mathbf y-\mathbf y'\|^{2M'-2\mathbf N}}{d(\mathbf x,\mathbf y)^{2M'-2\mathbf N}} + \frac{\|\mathbf y-\mathbf y'\|^{2\delta}}{d(\mathbf x,\mathbf y)^{2\delta}}\Big).
    \end{split}
\end{equation*}
The proof of \eqref{eq:kernel_assumption_2} is identical to that of \eqref{eq:kernel_assumption_1} and uses \eqref{eq:Holder_K33}.
\end{proof}
\section{Bessel potentials}

For a real number $s>0$ we set
\begin{align*}
    J^{\{s\}}(\mathbf{x})=\mathcal{F}^{-1}\big((1+\|\cdot\|^2)^{-s\slash 2}\big)(\mathbf{x}).
\end{align*}
By the gamma function identity we have
\begin{align*}
    (1+\|\xi\|^2)^{-s\slash 2}={\Gamma \Big(\frac{s}{2}\Big)^{-1}}\int_0^{\infty}e^{-t}e^{-t\|\xi\|^2}t^{s\slash 2}\,\frac{dt}{t},
\end{align*}
which leads us to
\begin{equation}\label{eq:subordination}
    J^{\{s\}}(\mathbf{x})={\Gamma \Big(\frac{s}{2}\Big)^{-1}}\int_0^{\infty}e^{-t}h_t(\mathbf{x})t^{s\slash 2}\,\frac{dt}{t}
\end{equation}
(see~\eqref{eq:heat_h_t}). The function $J^{\{s\}}$ is radial, positive and belongs $L^1(dw)$. Moreover, by~\eqref{eq:conv_def}, for all $s_1,s_2>0$ we have
\begin{equation}\label{eq:Bessel_semigroup}
    J^{\{s_1\}}*J^{\{s_2\}}=c_kJ^{\{s_1+s_2\}}.
\end{equation}

\subsection{Pointwise estimates for the Bessel integral kernel.
}
The following proposition is an easy consequence of \eqref{eq:subordination} and \eqref{eq:heat_h_t}.
\begin{proposition}\label{propo:bessel_estimate}
Let $M>0$. There is a constant $C=C_{s,M}>0$ such that
\begin{align*}
    0<J^{\{s\}}(\mathbf{x}) \leq C \begin{cases}
    \|\mathbf{x}\|^{s-\mathbf{N}} & \text{ if }\|\mathbf{x}\| \leq 1\slash 2, \quad 0<s<\mathbf N,\\
    -\ln \|\mathbf x\| & \text{ if }\|\mathbf{x}\| \leq 1\slash 2, \quad s=\mathbf N,\\
    1 & \text{ if } \| \mathbf x\|\leq 1/2, \quad s>\mathbf N,\\
    (1+\|\mathbf{x}\|^2)^{-M} & \text{ if }\|\mathbf{x}\|>1/2,  \quad 0<s.
    \end{cases}
\end{align*}
\end{proposition}

\begin{proposition}\label{propo:nonradial_optimal}
Assume that $M\geq 0$ and $|f(\mathbf{\mathbf{\mathbf{x}}})|\leq (1+\|\mathbf{\mathbf{\mathbf{x}}}\|)^{-M-\mathbf N}$ for all $\mathbf{x} \in \mathbb{R}^N$. Let $s>\frac{\mathbf N}{2}$. Then there is a constant $C>0$ such that for all $\mathbf{x},\mathbf{y} \in \mathbb{R}^N$ we have
$$ |(f*J^{\{s\}}*J^{\{s\}})(\mathbf{\mathbf{\mathbf{x}}},\mathbf{y})|\leq C\mathcal V(\mathbf{\mathbf{x}}, \mathbf y,1)^{-1} (1+d(\mathbf{x},\mathbf{y}))^{-M}.$$
\end{proposition}
\begin{proof}
Let  $U_j=B(0,2^j)\setminus B(0,2^{j-1})$ for $j \in \mathbb{Z}$. Set
$$ \phi_j(\mathbf{x})=J^{\{s\}}(\mathbf{x})\chi_{U_j}(\mathbf{x}).$$
Let $V_0=B(0,1)$, $V_{j}=U_{j}$ for $j > 0$. For $j \in \mathbb{N}_{0}$ we set
$$f_{j}(\mathbf{x})=f(\mathbf{x})\chi_{V_{j}}(\mathbf{x}).$$
Thanks to Proposition~\ref{propo:bessel_estimate}, for every $D>0$ large enough, we have
$$ \| f_j\|_{L^1(dw)}\leq C w(B(0,2^j))2^{-j(M+\mathbf N)}, \quad \| \phi_\ell \|_{L^\infty} \leq C_D\min(2^{-\ell D},2^{\ell(s-\mathbf{N})}) $$
with $C,C_{D}>0$ independent of $j \in \mathbb{N}_0$ and $\ell \in \mathbb{Z}$. Then
$$ (f*J^{\{s\}}*J^{\{s\}})(\mathbf{x},\mathbf{y})=\sum_{j \in \mathbb{N}_{0}\,,\ell,m \in \mathbb{Z}} (f_j*\phi_\ell*\phi_m)(\mathbf{x},\mathbf{y}).$$
Lemma 4.3 of \cite{DzH} asserts that
$$\Big(\int_{\mathbb{R}^N}\phi_m (\mathbf{z},\mathbf{y})^2\, dw (\mathbf{z})\Big)^{1/2} \leq CC_D\frac{2^{m\mathbf N}\min(2^{-mD},2^{m(s-\mathbf{N})})}{w(B(\mathbf{y},2^m))^{1/2}}\leq  CC_D\frac{\min(2^{m(\mathbf{N}-D)},2^{m(s-\mathbf{N}/2)})}{w(B(\mathbf{y},1))^{1/2}}.$$
Inequality (4.8) of \cite{DzH} asserts that
$$ \Big(\int_{\mathbb{R}^N}|f_j*\phi_\ell  (\mathbf{x},\mathbf{z})|^2\, dw (\mathbf{z})\Big)^{1/2} \leq  CC_D \frac{\min(2^{\ell(\mathbf{N}-D)},2^{\ell(s-\mathbf{N}/2)})}{w(B(\mathbf{x},1))^{1/2}}2^{-jM}. $$
Hence, by the Cauchy-Schwarz inequality, \begin{align*}
    |(f_j*\phi_\ell *\phi_m)(\mathbf{x},\mathbf{y})|&\leq CC_D^2 \frac{\min(2^{\ell(\mathbf{N}-D)},2^{\ell(s-\mathbf{N}/2)})}{w(B(\mathbf{x},1))^{1/2}} \frac{\min(2^{m(\mathbf{N}-D)},2^{m(s-\mathbf{N}/2)})}{w(B(\mathbf{y},1))^{1/2}}2^{-jM}\\&=:s_{j,m,\ell}(\mathbf{x},\mathbf{y}).
\end{align*}
Observe that $\sum_{j\geq 0, \, m,\ell\in \mathbb Z} s_{j,m,\ell}(\mathbf{x},\mathbf{y}) \leq C \mathcal V (\mathbf x,\mathbf y,1)^{-1}$. Therefore, to finish the proof, we consider $d(\mathbf x,\mathbf y)\geq 1$.
Recall that $(f_j*\phi_\ell *\phi_m)(\mathbf{x},\mathbf{y})=0$ if $d(\mathbf{x},\mathbf{y})>2^{j+\ell+m}$. Hence,
\begin{equation}
    \begin{split}
  &\sum_{j \geq 0, \, m,\ell \in \mathbb{Z}}| (f_j*\phi_\ell *\phi_m)(\mathbf{x},\mathbf{y}) |\leq \sum_{m,\ell \in \mathbb{Z}} \ \sum_{j \geq 0: 2^j\geq 2^{-\ell-m}d(\mathbf{x},\mathbf{y})}     s_{j,m,\ell}(\mathbf{x},\mathbf{y})\\
  &\leq C\sum_{m,\ell\in \mathbb{Z}}\frac{\min(2^{m(\mathbf N-D)},2^{m(s-\mathbf{N}/2)})\min(2^{\ell(\mathbf N-D)},2^{\ell(s-\mathbf{N}/2)}) \Big(d(\mathbf{x},\mathbf{y})2^{-\ell-m}\Big)^{-M}}{w(B(\mathbf{x},1))^{1/2}w(B(\mathbf{y},1))^{1/2}}\\
  &\leq C  \mathcal V(\mathbf{x},\mathbf y, 1)^{-1}d(\mathbf{x},\mathbf{y})^{-M}
    \end{split}
\end{equation}
\end{proof}

\subsection{\texorpdfstring{$L^1$}{L1} estimates for Bessel potentials}

\begin{lemma}\label{lem:axiom}
Let $M \geq 0$, $\kappa > 3\mathbf{N}/2 +M$, and $\gamma \geq 0$. Suppose that a measurable function $f:(0,\infty) \times \mathbb{R}^{N} \times \mathbb{R}^{N} \to \mathbb{C}$ satisfies
\begin{equation}\label{eq:kappa}
    |f(t,\mathbf{x},\mathbf{y})| \leq t^{-\gamma/2}{ V(\mathbf{x},\mathbf{y},\sqrt{t})^{-1}}\left(1+\frac{d(\mathbf{x},\mathbf{y})}{\sqrt{t}}\right)^{-\kappa}
\end{equation}
for all $\mathbf{x},\mathbf{y} \in \mathbb{R}^{N}$ and $t>0$. Then there is a constant $C=C_{M,\kappa,\gamma}>0$ such that for all $\mathbf{y} \in \mathbb{R}^N$ and $t>0$ we have
\begin{equation}\label{eq:first_axiom}
    \int_{\mathbb{R}^N} \frac{|f(t,\mathbf{x},\mathbf{y})|}{w(B(\mathbf{x},1))^{1/2}}(1+d(\mathbf{x},\mathbf{y}))^M\, dw(\mathbf{x}) \leq Ct^{-\gamma/2}\frac{(1+\sqrt{t})^{M+\mathbf N/2}}{w(B(\mathbf y,1))^{1/2}}.
\end{equation}
If additionally, for all $\mathbf{x},\mathbf{y},\mathbf{y}' \in \mathbb{R}^N$ and $\|\mathbf{y}-\mathbf{y}'\| \leq 1$ we have
\begin{equation}\label{eq:kappa_1}
    |f(t,\mathbf{x},\mathbf{y})-f(t,\mathbf{x},\mathbf{y}')| \leq \frac{\|\mathbf{y}-\mathbf{y}'\|}{\sqrt{t}} t^{-\gamma/2} V(\mathbf{x},\mathbf{y},\sqrt{t})^{-1}\left(1+\frac{d(\mathbf{x},\mathbf{y})}{\sqrt{t}}\right)^{-\kappa},
\end{equation}
then for any $\delta \in [0,1]$ there is a constant $C=C_{M,\kappa,\gamma,\delta}>0$ such that for all $\mathbf{y},\mathbf{y}' \in \mathbb{R}^N$ such that $\|\mathbf{y}-\mathbf{y}'\| \leq 1$ we have
\begin{equation}\label{eq:second_axiom}
\begin{split}
    &\int_{\mathbb{R}^N} \frac{|f(t,\mathbf{x},\mathbf{y})-f(t,\mathbf{x},\mathbf{y}')|}{w(B(\mathbf{x},1))^{1/2}}(1+d(\mathbf{x},\mathbf{y}))^M\, dw(\mathbf{x}) \\&\leq Ct^{-\gamma/2}\frac{\|\mathbf{y}-\mathbf{y}'\|^{\delta}}{t^{\delta/2}}\frac{(1+\sqrt{t})^{M+\mathbf N/2}}{w(B(\mathbf{y},1))^{1/2}}.
\end{split}
\end{equation}
\end{lemma}

\begin{proof}
 It follows from~\eqref{eq:growth} and~\eqref{eq:ball_ball} that $w(B(\mathbf x,1))^{-1/2}\leq C(B(\mathbf y, 1))^{-1/2} (1+d(\mathbf x,\mathbf y))^{\mathbf N/2}$.  We have
\begin{equation*}\label{eq:calculation_axion}
    \begin{split}
&\int_{\mathbb{R}^N}  \frac{|f(t,\mathbf{x},\mathbf{y})|}{w(B(\mathbf{x},1))^{1/2}}(1+d(\mathbf{x},\mathbf{y}))^M\, dw(\mathbf{x})
 \leq  C\int_{\mathbb{R}^N} \frac{|f(t,\mathbf{x},\mathbf{y})|}{ w(B(\mathbf{y},1))^{1/2}}(1+d(\mathbf{x},\mathbf{y}))^{M+\mathbf N/2} \, dw(\mathbf{x})\\
&\leq C\int_{\mathbb{R}^N} \frac{|f(t,\mathbf{x},\mathbf{y})|}{w(B(\mathbf{y},1))^{1/2}}\Big(1+\frac{d(\mathbf{x},\mathbf{y})}{\sqrt{t}}\Big)^{M+\mathbf N/2} (1+\sqrt{t})^{M+\mathbf N/2}\, dw(\mathbf{x})\\
&\leq Ct^{-\gamma/2}(1+\sqrt{t})^{M+\mathbf N/2} w(B(\mathbf{y},1))^{-1/2}.
    \end{split}
\end{equation*}
where in the last estimate we have used the assumption $\kappa>3\mathbf{N}/2+M$.

In order to prove~\eqref{eq:second_axiom}  we note that for all $\mathbf{x},\mathbf{y},\mathbf{y}' \in \mathbb{R}^N$ such that $\|\mathbf{y}-\mathbf{y}'\| \leq 1$ we have
\begin{align*}
    (1+d(\mathbf{x},\mathbf{y}))^{M} \sim  (1+d(\mathbf{x},\mathbf{y}'))^{M} \ \ \text{ and } \ \ w(B(\mathbf{y},1)) \sim w(B(\mathbf{y}',1)).
\end{align*}
Consequently, by~\eqref{eq:first_axiom},
\begin{equation}\label{eq:first_axiom_with_y'}
     \int_{\mathbb{R}^N} \frac{|f(t,\mathbf{x},\mathbf{y}')|}{w(B(\mathbf{x},1))^{1/2}}(1+d(\mathbf{x},\mathbf{y}))^M\, dw(\mathbf{x}) \leq Ct^{-\gamma/2}\frac{(1+\sqrt{t})^{M+\mathbf N/2}}{w(B(\mathbf y,1))^{1/2}}
\end{equation}
Furthermore, if $0<\|\mathbf y-\mathbf y'\|\leq 1$, then,  by~\eqref{eq:first_axiom}  applied to
\begin{align*}
    \widetilde{f}(t,\mathbf{x},\mathbf{y})=\frac{1}{\|\mathbf{y}-\mathbf{y}'\|}(f(t,\mathbf{x},\mathbf{y})-f(t,\mathbf{x},\mathbf{y}'))
\end{align*}
with $\widetilde{\gamma}=\gamma+1$, we get
\begin{equation}\label{eq:lipschitz_1}
    \begin{split}
         &\int_{\mathbb{R}^N} w(B(\mathbf{x},1))^{-1/2}|f(t,\mathbf{x},\mathbf{y})-f(t,\mathbf{x},\mathbf{y}')|(1+d(\mathbf{x},\mathbf{y}))^M\, dw(\mathbf{x}) \\&\leq Ct^{-\gamma/2}\frac{\|\mathbf{y}-\mathbf{y}'\|}{\sqrt{t}}(1+\sqrt{t})^{M+\mathbf N/2}w(B(\mathbf{y},1))^{-1/2}.
    \end{split}
\end{equation}
Now~\eqref{eq:second_axiom} is a consequence of~\eqref{eq:first_axiom},~\eqref{eq:first_axiom_with_y'}, and~\eqref{eq:lipschitz_1}.
\end{proof}

\begin{proposition}\label{propo:Bessel_with_der}
Let $\delta \in [0,1]$, $M \geq 0$, $\beta,\beta' \in \mathbb{N}_0^{N}$ and $s>|\beta|+|\beta'|$. There is a constant $C=C_{M,\beta,\beta',s}>0$ such that for all $\mathbf{y} \in \mathbb{R}^N$ we have
\begin{equation}\label{eq:first_1_axiom}
    \int_{\mathbb{R}^N}w(B(\mathbf{x},1))^{-1/2}|\partial_{\mathbf x}^{\beta}\partial_{\mathbf{y}}^{\beta'}J^{\{s\}}(\mathbf{x},\mathbf{y})|(1+d(\mathbf{x},\mathbf{y}))^M\, dw(\mathbf{x}) \leq \frac{C}{w(B(\mathbf{y},1))^{1/2}}.
\end{equation}
Moreover, if we assume that $s>\delta+|\beta|+|\beta'|$, then there is a constant $C=C_{M,\beta,\beta',s,\delta}$ such that for all $\mathbf{y},\mathbf{y}' \in \mathbb{R}^N$ such that $\|\mathbf{y}-\mathbf{y}'\| \leq 1$ we have
\begin{equation}\label{eq:second_1_axiom}
    \int_{\mathbb{R}^N}w(B(\mathbf{x},1))^{-1/2}\Big|\partial^{\beta}_{\mathbf x}\partial^{\beta'}_{\mathbf{y}}J^{\{s\}}(\mathbf{x},\mathbf{y})-\partial^{\beta}_{\mathbf x}\partial_{\mathbf{y}}^{\beta'}J^{\{s\}}(\mathbf{x},\mathbf{y}')\Big| (1+d(\mathbf{x},\mathbf{y}))^M\, dw(\mathbf{x}) \leq \frac{C\| \mathbf{y}-\mathbf{y}'\|^\delta}{w(B(\mathbf{y},1))^{1/2}}.
\end{equation}
\end{proposition}

\begin{proof}
Note that if $s>\gamma$, then there is a constant $C=C_{M,s}>0$ such that
\begin{equation}\label{eq:sub_app}
    \int_0^{\infty}e^{-t}t^{-\gamma/2}(1+t^{\mathbf{N}/4})(1+\sqrt{t})^{M}t^{s/2}\frac{dt}{t} \leq C.
\end{equation}
On the other hand, by Theorem~\ref{teo:heat} there is a constant $C=C_{\beta,\beta'}>0$ such that
\begin{equation}\label{eq:f_app}
    f(t,\mathbf{x},\mathbf{y})=C^{-1}\partial_{\mathbf{x}}^{\beta}\partial_{\mathbf{y}}^{\beta'}h_{t}(\mathbf{x},\mathbf{y})
\end{equation}
satisfies~\eqref{eq:kappa} with $\gamma=|\beta|+|\beta'|$ and $\kappa=2\mathbf{N}+M$. Therefore,~\eqref{eq:first_1_axiom} is a consequence of Lemma~\ref{lem:axiom},~\eqref{eq:sub_app}, and~\eqref{eq:subordination}. Similarly, by Theorem~\ref{teo:heat} and mean value theorem, there is a constant $C>0$ such  $f(t,\mathbf{x},\mathbf{y})$ defined in~\eqref{eq:f_app} satisfies~\eqref{eq:kappa_1} with $\gamma=|\beta|+|\beta'|$ and $\kappa=2\mathbf{N}+M$. Consequently,~\eqref{eq:second_1_axiom} follows by Lemma~\ref{lem:axiom},~\eqref{eq:sub_app}, and~\eqref{eq:subordination}.
\end{proof}

\begin{proposition}\label{propo:nonradial_optimal_lip}
Assume that $M\geq 0$ and $|f(\mathbf{\mathbf{\mathbf{x}}})|\leq (1+\|\mathbf{\mathbf{\mathbf{x}}}\|)^{-M-\mathbf N}$ for all $\mathbf{x} \in \mathbb{R}^N$. Let $s>\frac{\mathbf N}{2}$ and $0<\delta<2s-\mathbf N$, $0<\delta\leq 1$. Then there is a constant $C>0$ such that for all $\mathbf{x},\mathbf{y},\mathbf{y}' \in \mathbb{R}^N$, $\|\mathbf y-\mathbf y'\|\leq 1$ we have
\begin{equation*}
     |(f*J^{\{s\}}*J^{\{s\}})(\mathbf{\mathbf{\mathbf{x}}},\mathbf{y})-(f*J^{\{s\}}*J^{\{s\}})(\mathbf{\mathbf{\mathbf{x}}},\mathbf{y'})|\leq C\mathcal V(\mathbf x,\mathbf y,1)^{-1}\|\mathbf y-\mathbf y'\|^\delta (1+d(\mathbf{x},\mathbf{y}))^{-M}.
\end{equation*}
\end{proposition}
\begin{proof}
Take $s_1>\mathbf N/2$ and $s_2>0$ such that $2s=2s_1+s_2$, $s_2>\delta$. Then
\begin{align*}
    f*J^{\{2s\}}=c_k^{-1}f*J^{\{2s_1\}}*J^{\{s_2\}}.
\end{align*}
By Proposition \ref{propo:nonradial_optimal}, for all $\mathbf{x},\mathbf{z} \in \mathbb{R}^N$ we have
$$ |(f*J^{\{s_1\}}*J^{\{s_1\}})(\mathbf{\mathbf{\mathbf{x}}},\mathbf{z})|\leq C\mathcal V(\mathbf x,\mathbf z,1)^{-1} (1+d(\mathbf{x},\mathbf{z}))^{-M}. $$
Applying \eqref{eq:second_1_axiom} we obtain
\begin{equation*}
    \begin{split}
     (1&+d(\mathbf x,\mathbf y))^M   |f*J^{\{2s\}}(\mathbf x,\mathbf y)-f*J^{\{2s\}}(\mathbf x,\mathbf y')|\\
        &\leq C \int_{\mathbb{R}^N} (1+d(\mathbf x,\mathbf z))^M(1+d(\mathbf z,\mathbf y))^M|f*J^{\{2s_1\}}(\mathbf x,\mathbf z)|\cdot |J^{\{s_2\}}(\mathbf z,\mathbf y)-J^{\{s_2\}}(\mathbf z,\mathbf y')|\, dw(\mathbf z)\\
        &\leq C\int_{\mathbb{R}^N} \mathcal V(\mathbf x,\mathbf z,1)^{-1} (1+d(\mathbf z,\mathbf y))^M|J^{\{s_2\}}(\mathbf z,\mathbf y)-J^{\{s_2\}}(\mathbf z,\mathbf y')|\, dw(\mathbf z)\\
        &\leq C \mathcal V(\mathbf x,\mathbf y,1)^{-1} \| \mathbf y-\mathbf y'\|^\delta.
    \end{split}
\end{equation*}
\end{proof}

\section{Estimates of kernels}

\begin{proposition}\label{propo:est_phi}
Let $s$ be a positive integer such that  $2s>\mathbf N$. Assume that a function $\varphi\in C^{2s}(\mathbb R^N)$ satisfies
$$| \partial^\beta \varphi(\mathbf x)|\leq C_\beta (1+\|\mathbf x\|)^{-M-\mathbf N} \quad {\rm for }\  |\beta|\leq 2s,$$
for certain $M>0$. Then there is a constant $C>0$ which depends on $C_\beta$, $M$ and $s$, such that
\begin{equation}\label{2wagi1} |\varphi_t(\mathbf x,\mathbf y)|\leq C \mathcal V(\mathbf x,\mathbf y, t)^{-1}\Big(1+\frac{d(\mathbf x,\mathbf y)}{t}\Big)^{-M}.
\end{equation}
Moreover, if $0<\delta\leq 1$ is such that $\delta<2s-\mathbf N$, then there is a constant $C>0$ such that  for $\| \mathbf y-\mathbf y'\|\leq t$ one has
\begin{equation}\label{2wagi2} |\varphi_t(\mathbf x,\mathbf y)-\varphi_t(\mathbf x,\mathbf y')|\leq C \mathcal V(\mathbf x,\mathbf y, t)^{-1}\frac{\| \mathbf y-\mathbf y'\|^\delta}{t^\delta}\Big(1+\frac{d(\mathbf x,\mathbf y)}{t}\Big)^{-M}.
\end{equation}
If additionally $2s>\mathbf N+1$, then
\begin{equation}\label{2wagi3} |\partial_{j,\mathbf x}\varphi_t(\mathbf x,\mathbf y)|\leq C  \mathcal V(\mathbf x,\mathbf y, t)^{-1}t^{-1}\Big(1+\frac{d(\mathbf x,\mathbf y)}{t}\Big)^{-M}.
\end{equation}
Moreover, if  $\delta>0$ is such that $0<\delta\leq 1$, $\delta<2s-\mathbf N-1$, then there is $C>0$ such that for $\| \mathbf y-\mathbf y'\|\leq t$, one has
\begin{equation}\label{2wagi4} |\partial_{j,\mathbf x}\varphi_t(\mathbf x,\mathbf y)-\partial_{j,\mathbf x}\varphi_t(\mathbf x,\mathbf y')|\leq C  \mathcal V(\mathbf x,\mathbf y, t)^{-1}t^{-1}\frac{\| \mathbf y-\mathbf y'\|^\delta}{t^\delta}\Big(1+\frac{d(\mathbf x,\mathbf y)}{t}\Big)^{-M}.
\end{equation}
Similarly, if  $\|\mathbf x-\mathbf x'\|\leq t$, then
\begin{equation}\label{2wagi5} |\partial_{j,\mathbf x}\varphi_t(\mathbf x,\mathbf y)-\partial_{j,\mathbf x}\varphi_t(\mathbf x',\mathbf y)|\leq C  \mathcal V(\mathbf x,\mathbf y, t)^{-1}t^{-1}\frac{\| \mathbf x-\mathbf x'\|^\delta}{t^\delta}\Big(1+\frac{d(\mathbf x,\mathbf y)}{t}\Big)^{-M}.
\end{equation}

\end{proposition}
\begin{proof}
It suffices to prove \eqref{2wagi1}--\eqref{2wagi5} for $t=1$ and then use scaling.

Let {$f=(I-\Delta_k)^s \varphi$}. Then $\varphi=c_k^{-1}f*J^{\{s\}}*J^{\{s\}}$. Applying  Proposition~\ref{propo:nonradial_optimal} we obtain \eqref{2wagi1}.

We now turn to prove \eqref{2wagi2}. Fix $0<\delta\leq 1$, $\delta<2s-\mathbf N$.
Let $s_1>\mathbf N/2$ and {$s_2>\delta$} be such that $2s=2s_1+s_2$.   Then $\varphi=c_k^{-1}f*J^{\{2s_1\}}*J^{\{s_2\}}.$ By Proposition~\ref{propo:nonradial_optimal},
\begin{equation}\label{eq:fJ2s} | f*J^{\{2s_1\}}(\mathbf x,\mathbf z)|\leq C\mathcal V(\mathbf{x},\mathbf{z},1)^{-1}\Big(1+d(\mathbf x,\mathbf z)\Big)^{-M}.
\end{equation}
 Applying \eqref{eq:fJ2s} we get
\begin{equation*}
    \begin{split}
        |\varphi_t &(\mathbf x,\mathbf y)-\varphi_t(\mathbf x,\mathbf y')|(1+d(\mathbf x,\mathbf y))^M\\
        & \leq C\int_{\mathbb{R}^N}(1+d(\mathbf x,\mathbf z))^M|(f*J^{\{2s_1\}})(\mathbf x,\mathbf z)|(1+d(\mathbf z,\mathbf y))^M|J^{\{s_2\}}(\mathbf z,\mathbf y)-J^{\{s_2\}}(\mathbf z,\mathbf y')|\, dw(\mathbf z)\\
        &\leq C \int_{\mathbb{R}^N}\mathcal V(\mathbf{x},\mathbf{z},1)^{-1} (1+d(\mathbf z,\mathbf y))^M |J^{\{s_2\}}(\mathbf z,\mathbf y)-J^{\{s_2\}}(\mathbf z,\mathbf y')|\, dw(\mathbf z)\\
        &\leq  C\mathcal V(\mathbf{x},\mathbf{y},1)^{-1} \| \mathbf y-\mathbf y'\|^{\delta},
    \end{split}
\end{equation*}
{where in the last inequality we have used \eqref{eq:second_1_axiom} with $\beta=\beta'=\mathbf{0}$. }

{In order to prove  \eqref{2wagi3}--\eqref{2wagi4}, we fix $0<\delta<2s-\mathbf N-1$, $\delta\leq 1$, and  take $s_1$ and $s_2$ such that $2s_1>\mathbf N+\delta$, $s_2>1$, $2s=2s_1+s_2$. We write
$  \varphi=c_k^{-1}J^{\{s_2\}}*(f*J^{\{2s_1\}})$, where $f=(I-\Delta_k)^{s}\varphi$. Then
\begin{equation}\label{eq:decomp}  \partial_{j,\mathbf x}\phi(\mathbf x,\mathbf y)=c_k^{-1}\int_{\mathbb{R}^N}\partial_{j,\mathbf x}J^{\{s_2\}}(\mathbf x,\mathbf z)(f*J^{\{{2s_1}\}})(\mathbf z,\mathbf y)\, dw(\mathbf z).
\end{equation}
 Now, having \eqref{eq:translation} in mind,  we use Propositions  \ref{propo:nonradial_optimal}, ~\ref{propo:Bessel_with_der}, and  \ref{propo:nonradial_optimal_lip}, and proceed as in the proofs of  \eqref{2wagi1} and \eqref{2wagi2} to obtain  \eqref{2wagi3}--\eqref{2wagi4}.}

The proof of~\eqref{2wagi5} is identical, however, this time,  for fixed $0<\delta<2s- \mathbf N-1$, we use the formula~\eqref{eq:decomp}  with $2s_1>\mathbf N$, $s_2>1+\delta$, $2s=2s_1+s_2$.
\end{proof}

For  $\alpha\in R$, set
\begin{equation}\label{eq:K_alpha}
K^{\{\alpha\}}(t, \mathbf x,\mathbf y)= t\frac{\phi_t(\mathbf x,\mathbf y)-\phi_t(\sigma_{\alpha}(\mathbf x),\mathbf y)}{\langle \alpha, \mathbf x\rangle}.
\end{equation}

\begin{proposition}\label{propo:Kernel_K}  Let $s$ be a positive integer such that $2s>\mathbf N+1$. Assume that  $\phi\in C^{2s}(\mathbb R^N)$ satisfies
\begin{equation}
   | \partial^\beta \phi(\mathbf x)|\leq (1+\|\mathbf x\|)^{-M-\mathbf N} \quad\text{for all}\quad |\beta|\leq 2s,
\end{equation}
for certain $M>\lfloor \mathbf N \rfloor +1$.
  Then there is a constant $C>0$ such that for all $\mathbf{x},\mathbf{y} \in \mathbb{R}^N$ and $t>0$ we have
\begin{equation}\label{eq:estimate_K}| K^{\{\alpha\}}(t,\mathbf x,\mathbf y)|\leq C \mathcal V(\mathbf x,\mathbf y, t)^{-1} \Big(1+\frac{d(\mathbf x,\mathbf y)}{t}\Big)^{-M}.
\end{equation}
Moreover, there is a constant $C>0$ and $0<\delta\leq 1$ such that for all $\mathbf{x},\mathbf{y},\mathbf{y}' \in \mathbb{R}^N$ and $\| \mathbf y-\mathbf y'\|\leq t$ one has
\begin{equation}\label{eq:Holder_K} |K^{\{\alpha\}}(t,\mathbf x,\mathbf y)-K^{\{\alpha\}}(t,\mathbf x,\mathbf y')|\leq C \frac{\| \mathbf y-\mathbf y'\|^\delta}{t^\delta} \mathcal V(\mathbf x,\mathbf y, t)^{-1}\Big(1+\frac{d(\mathbf x,\mathbf y)}{t}\Big)^{-M},
\end{equation}
\begin{equation}\label{eq:Holder_K2} |K^{\{\alpha\}}(t,\mathbf y,\mathbf x)-K^{\{\alpha\}}(t,\mathbf y',\mathbf x)|\leq C \frac{\| \mathbf y-\mathbf y'\|^{\delta}}{t^\delta} \mathcal V(\mathbf x,\mathbf y ,t)^{-1}\Big(1+\frac{d(\mathbf x,\mathbf y)}{t}\Big)^{-M}.
\end{equation}
\end{proposition}

\begin{proof}
Recall that $\| \mathbf x-\sigma_{\alpha}(\mathbf x)\|=\sqrt{2}|\langle \alpha,\mathbf x\rangle|$ (see~\eqref{eq:refl}). So, if $|\langle \alpha,\mathbf x\rangle |<t$, then \eqref{eq:estimate_K} follows from \eqref{2wagi2} (with  $ \delta=1$). Otherwise we apply \eqref{2wagi1} to obtain \eqref{eq:estimate_K}, because, by~\eqref{eq:d}, $d(\mathbf x,\mathbf y)=d(\sigma_{\alpha}(\mathbf x),\mathbf y)$.

{In order to prove~\eqref{eq:Holder_K} take $s_1,s_2>0$ such that $2s=2s_1+s_2$,   $2s_1>\mathbf N+1$. Fix $0<\delta<s_2$.
Set
\begin{equation*}\begin{split}& f(\mathbf x)=(I-\Delta_{k})^{s}\phi (\mathbf x),\quad \phi^{\{1\}}=J^{\{2s_1\}}*f,\\
&  K^{\{1,\alpha\}}(t,\mathbf x,\mathbf z)=\frac{t}{\langle \alpha,\mathbf x\rangle }\Big( \phi^{\{1\}}_t(\mathbf x,\mathbf z)-\phi^{\{1\}}_t( \sigma_\alpha (\mathbf x),\mathbf z)\Big).\\
\end{split}
\end{equation*}
By~\eqref{eq:Bessel_semigroup} and Propositions~\ref{propo:nonradial_optimal_lip} and \ref{propo:nonradial_optimal} we have
\begin{equation*}
    \begin{split}
       |K^{\{1,\alpha\}}(t,\mathbf x,\mathbf z)|
       &\leq C \begin{cases}\frac{t}{|\langle \alpha,\mathbf x\rangle|}\frac{\| \mathbf x-\sigma_{\alpha}( \mathbf x) \|}{t} \mathcal V(\mathbf x,\mathbf z,t)^{-1}\Big(1+\frac{d(\mathbf x,\mathbf z)}{t}\Big)^{-M}\quad &\text{\rm if } \|\mathbf x-\sigma_\alpha(\mathbf x)\|\leq t\\
       \mathcal V(\mathbf x,\mathbf z,t)^{-1} \Big(1+\frac{d(\mathbf x,\mathbf z)}{t}\Big)^{-M}\quad &\text{\rm if } \|\mathbf x-\sigma_\alpha(\mathbf x)\|>t       \end{cases} \\
      &\leq C \mathcal V(\mathbf x,\mathbf z,t)^{-1} \Big(1+\frac{d(\mathbf x,\mathbf z)}{t}\Big)^{-M}.
    \end{split}
\end{equation*}
Finally, for $\|\mathbf y-\mathbf y'\|<t$, applying  Proposition~\ref{propo:Bessel_with_der}, we arrive to
\begin{equation*}
    \begin{split}
   (1+&d(\mathbf x,\mathbf y))^M     |K^{\{\alpha\}}(t,\mathbf x,\mathbf y)-K^{\{\alpha\}}(t,\mathbf x,\mathbf y')|\\
   &\leq C \Big|\int(1+d(\mathbf x,\mathbf z))^M(1+d(\mathbf z,\mathbf y))^M K^{\{1,\alpha\}}(t,\mathbf x,\mathbf z)\Big((J^{\{s_2\}})_t(\mathbf z,\mathbf y)-
        (J^{\{s_2\}})_t(\mathbf z,\mathbf y')\Big)\, dw(\mathbf z)\Big|  \\
                &\leq C  \frac{\|\mathbf y-\mathbf y'\|^\delta}{t^\delta}\mathcal V(\mathbf x,\mathbf y,t)^{-1},
    \end{split}
\end{equation*}
which proves \eqref{eq:Holder_K}.
}

We now turn to prove \eqref{eq:Holder_K2}.
We may assume that $\|\mathbf{y}-\mathbf{y}'\|<\frac{1}{8}t$, otherwise, for $t/8\leq \|\mathbf y-\mathbf y'\|\leq t$, the inequality ~\eqref{eq:Holder_K2} is a consequence of~\eqref{eq:estimate_K}. We consider two cases.

{\bf Case 1:}   $ \|\mathbf{y}-\sigma_{\alpha}(\mathbf{y})\|>t/2$.  Then $\sqrt{2}|\langle \mathbf{y},\alpha \rangle|=\|\mathbf{y}-\sigma_{\alpha}(\mathbf{y})\|>t/2$ and  $\sqrt{2}|\langle \mathbf y',\alpha\rangle|=\|\mathbf{y}'-\sigma_{\alpha}(\mathbf{y}')\|>t/4$.  So by~\eqref{2wagi2} (with $\delta=1$) we get
\begin{equation*}\label{eq:case1_1}
\begin{split}    |K^{\{\alpha\}}(t,\mathbf{y},\mathbf{x})-
K^{\{\alpha\}}(t,\mathbf{y'},\mathbf{x})| &\leq \frac{t}{|\langle \mathbf{y}',\alpha \rangle|}|\phi_{t}(\mathbf{y},\mathbf{x})-\phi_{t}(\mathbf{y}',\mathbf{x})|\\
&+\frac{t}{|\langle \mathbf{y}',\alpha \rangle|}|\phi_{t}(\sigma_{\alpha}(\mathbf{y}),\mathbf{x})-\phi_{t}(\sigma_{\alpha}(\mathbf{y}'),\mathbf{x})|\\
&+\frac{t}{|\langle \mathbf{y},\alpha \rangle\langle \mathbf{y}',\alpha\rangle|}|\langle \mathbf{y}-\mathbf{y}',\alpha \rangle |(|\phi_{t}(\mathbf{y},\mathbf{x})|+|\phi_{t}(\sigma_{\alpha}(\mathbf{y}),\mathbf{x})|)\\&\leq C\frac{\|\mathbf{y}-\mathbf{y}'\|}{t}\mathcal V(\mathbf y, \mathbf{x},t)^{-1}\Big(1+\frac{d(\mathbf{x},\mathbf{y})}{t}\Big)^{-M}.
\end{split}
\end{equation*}

{\bf Case 2:}   $ \|\mathbf{y}-\sigma_{\alpha}(\mathbf{y})\|\leq t/2$.
For $\tau \in [0,1]$ we set $\mathbf y(\tau)=\tau (\mathbf y-\sigma_\alpha (\mathbf y))+\sigma_\alpha(\mathbf y)$. Note that
\begin{equation}\label{eq:K_t_int}
    \begin{split}
        K^{\{\alpha\}}(t,\mathbf y,\mathbf x)
        &=\frac{t}{\langle \alpha,\mathbf y\rangle} \int_0^1\frac{d}{d\tau}
  \Big\{ \phi_t(\mathbf y(\tau),\mathbf x)\Big\}\, d\tau
  =t\int_0^1 \langle{\nabla_{(1)}} (\phi_t)(\mathbf{y}(\tau),\mathbf x),\alpha\rangle\, d\tau,
  \end{split}
\end{equation}
{where the symbol $\nabla_{(1)}$ denotes the gradient with respect to the first $N$-variables.}
 Observe that $\|\mathbf y(\tau)-(\mathbf y')(\tau)\| \leq C\| \mathbf y-\mathbf y'\|\leq Ct$.
Hence, by \eqref{2wagi5} combined with \eqref{eq:K_t_int}, we obtain  that there is $0<\delta\leq 1$ such that
\begin{equation}\label{eq:holder33}
  \begin{split}
    |K^{\{\alpha\}}(t,\mathbf y,\mathbf x)-K^{\{\alpha\}}(t,\mathbf y',\mathbf x)|
    \leq C_M\frac{\| \mathbf y-\mathbf y'\|^\delta}{t^\delta} \int_0^1     V(\mathbf y(\tau),\mathbf x,t)^{-1} \Big(1+\frac{d(\mathbf y(\tau),\mathbf x)}{t}\Big)^{-M} \, d\tau .\\
    \end{split}
\end{equation}
The assumption
$\|\mathbf{y}-\sigma_{\alpha}(\mathbf{y})\|\leq t/2$ implies
\begin{equation}\label{eq:distances}
    d(\mathbf y(\tau),\mathbf x)\geq d(\mathbf y,\mathbf x)-d(\mathbf y(\tau),\mathbf y)\geq d(\mathbf y ,\mathbf x)-\|\mathbf y(\tau)-\mathbf y\|\geq d(\mathbf y,\mathbf x)-\frac{1}{2}t.
\end{equation}
Moreover, $\mathcal V(\mathbf y(\tau),\mathbf x,t)\sim \mathcal V(
\mathbf y,\mathbf x,t)$. So, from \eqref{eq:holder33} and \eqref{eq:distances}, we conclude  \eqref{eq:Holder_K2}.
\end{proof}

\begin{corollary}\label{coro:nabla_k}
Under assumptions of Proposition~\ref{propo:Kernel_K} the integral kernel  $t \nabla_{k,\mathbf x} \phi_t(\mathbf x,\mathbf y)$ associated with the square function $S_{\nabla_k,\phi}$ satisfies  the conditions \eqref{eq:estimate_K11}-\eqref{eq:Holder_K33}.
\end{corollary}
\begin{proof}
The corollary is consequence of \eqref{eq:Dunkl_op}, Propositions \ref{propo:est_phi} and \ref{propo:Kernel_K},   it suffices to take $M'>\mathbf N$ and $\delta>0$ (small enough) such that $M=M'+\delta$.
\end{proof}

\section{\texorpdfstring{$L^2(dw)$}{L2}-bounds for square functions}

{In this section we assume that  $\phi\in C^{2s}(\mathbb R^N)$, for certain $s$ being a positive integer such that $2s>\mathbf N+1$ and satisfies
\begin{equation}\label{MainAssumption}
   | \partial^\beta \phi(\mathbf x)|\leq (1+\|\mathbf x\|)^{-M-\mathbf N} \text{ for all } |\beta|\leq 2s
\end{equation}
for certain $M>\lfloor \mathbf N \rfloor +1$.}

By  straightforward calculations (see~\cite[Lemma 4.4]{Graczyk} or~\cite[Lemma 3.1]{Velicu1}) we have
\begin{equation}\label{form_Gamma1} \Gamma(f,g)(\mathbf x)=\langle \nabla f(\mathbf x),\nabla g(\mathbf x)\rangle + \sum_{\alpha \in R}\frac{k(\alpha)}{2}\frac{(f(\mathbf x)-f(\sigma_\alpha(\mathbf x)))(\overline{ g(\mathbf x)-g(\sigma_\alpha(\mathbf x)}))}{\langle \alpha,\mathbf x\rangle^2}.
\end{equation}
Observe that
$\Gamma(\phi_t*f,\phi_t*f)(\mathbf x)$
is the sum of non-negative functions. Using ~\eqref{form_Gamma},~\eqref{g_carre},~\eqref{eq:by_parts}, and Plancherel's formula~\eqref{eq:Plancherel} together with~\eqref{eq:der_transform} we get
\begin{equation}\label{eq:Calderon}
    \begin{split}
        \| \mathfrak g_{\Gamma,\phi} & (f)\|_{L^2(dw)}^2\\ &=\frac{1}{2}\int_{0}^{\infty}t^2\int_{\mathbb{R}^N} \Delta_k((\phi_t*f)(\overline{\phi_t*f}))- \phi_t*f\Delta_k (\overline{\phi_t*f})- \overline{\phi_t*f}\Delta_k ({\phi_t*f}) \, dw\, \frac{dt}{t}\\
        &= \int_{0}^{\infty}  \int_{\mathbb{R}^N}t^2|\nabla_k (\phi_t*f)(\mathbf x)|^2dw(\mathbf x)\frac{dt}{t}\\
      & =\int_0^\infty \int_{\mathbb{R}^N}t^2 \| \xi\|^2 |\mathcal F\phi(t\xi)|^2 |\mathcal F f(\xi)|^2 dw(\xi)\frac{dt}{t}\\
      &=\int_{\mathbb{R}^N}c_\phi(\xi) |\mathcal Ff(\xi)|^2\, dw(\xi),
    \end{split}
\end{equation}
where
\begin{equation}\label{eq:c_phi} c_\phi(\xi)=\int_0^\infty t^2 \|\xi\|^2|\mathcal F\phi(t\xi)|^2\, \frac{dt}{t} .
\end{equation}
{ Observe that thanks to \eqref{transform_bound} of Proposition \ref{prop:ord_implies_Dunkl} (with $\beta'=0$ and $\ell =2$) the  function $c_\phi$ is  bounded and  homogeneous of degree 0.  }
Using the Plancherel identity~\eqref{eq:Plancherel},  we obtain
\begin{equation}\label{eq:L2g}
     \|\mathfrak g_{\Gamma,\phi}f\|_{L^2(dw)}\leq C\| f\|_{L^2(dw)}.
\end{equation}
For $\alpha \in R$ let
$$ S_{K^{\{\alpha\}}}f(\mathbf x)= \Big(\int_0^\infty | K_t^{\{\alpha\}}f(\mathbf x)|^2 \frac{dt}{t}\Big)^{1/2}, $$
where $K_t^{\{\alpha\}}$ is defined by~\eqref{eq:K_alpha}. By~\eqref{form_Gamma1} we have
\begin{equation}
    \mathfrak g_{\Gamma,\phi}(f)(\mathbf x)^2=S_{\nabla, \phi}f(\mathbf x)^2+\sum_{\alpha\in  R}\frac{k(\alpha)}{2}S_{K^{\{\alpha\}}} f(\mathbf x)^2.
\end{equation}
Since
$$ tT_j\phi_t*f(\mathbf x)-t\partial_j \phi_t*f(\mathbf x)=\sum_{\alpha\in R} \frac{k(\alpha)}{2}\alpha_j K_t^{\{\alpha\}}f(\mathbf x), $$
 we obtain the pointwise bounds
\begin{equation}\label{eq:dom1} S_{\nabla,\phi}f(\mathbf x)+ S_{\nabla_k,\phi}f(\mathbf x)+ \sum_{\alpha\in R}\frac{k(\alpha)}{2}S_{K^{\{\alpha\}}}f(\mathbf x)\leq C \mathfrak g_{\Gamma,\phi}f(\mathbf x),
\end{equation}
\begin{equation}\label{eq:dom2} \mathfrak g_{\Gamma, \phi}(f)(\mathbf x)\leq C\Big(S_{\nabla_k,\phi}f(\mathbf x)+\sum_{\alpha\in R} \frac{k(\alpha)}{2} S_{K^{\{\alpha\}}} f(\mathbf x)\Big).
\end{equation}
Consequently, by \eqref{eq:L2g} and \eqref{eq:dom1},
\begin{equation}\label{eq:globaL2}
   \|S_{\nabla,\phi}f\|_{L^2(dw)}+\| S_{\nabla_k,\phi}f\|_{L^2(dw)}+ \sum_{\alpha\in R}k(\alpha)\|S_{K^{\{\alpha\}}}f\|_{L^2(dw)}\leq C\| f\|_{L^2(dw)}.
\end{equation}
{Assume that   $\psi\in C^{2s} (\mathbb R^N)$ satisfying  \eqref{MainAssumption}} is  such that $\int_{\mathbb{R}^N} \psi(\mathbf x)\, dw(\mathbf x)=0$. Let
$$ \tilde c_\psi(\xi)=\int_0^\infty | \mathcal F \psi(t\xi)|^2\frac{dt}{t}.$$
{From Proposition~\ref{prop:ord_implies_Dunkl}  we conclude that  $\tilde c_{\psi}(\xi)$ is a bounded homogeneous of degree 0 function.} It can be proved using   the Dunkl transform and the Plancherel identity  (cf.~\eqref{eq:Calderon}) that
\begin{equation}\label{eq:globalL2_1}
    \|S_\psi (f)\|_{L^2(dw)}^2=\int| \mathcal Ff(\mathbf \xi)|^2\tilde c_\psi(\xi)\, dw(\xi)\leq C\| f\|^2_{L^2(dw)}.
\end{equation}
We finish this section by writing the following easily proved identities for $f,g \in L^2(dw)$ (cf.~\eqref{eq:Calderon} and~\eqref{eq:globalL2_1}):
\begin{equation}
    \begin{split}
       \int_{\mathbb{R}^N} \int_0^\infty t^2 \Gamma (\phi_t*f,\phi_t*g)(\mathbf x)\frac{dt}{t}\, dw(\mathbf x)
       &=\int_{\mathbb{R}^N} \mathcal F f(\xi)\overline{\mathcal F g(\xi)}c_\phi(\xi)\, dw(\xi),
    \end{split}
\end{equation}
\begin{equation}\label{eq:polar}
      \int_{\mathbb{R}^N} \int_0^\infty t^2 \langle  \nabla_k(\phi_t*f)(\mathbf x),\nabla_k(\phi_t*g)(\mathbf x)\rangle \frac{dt}{t}\, dw(\mathbf x)= \int_{\mathbb{R}^N} \mathcal F f(\xi)\overline{\mathcal F g(\xi)}c_\phi(\xi)\, dw(\xi),
\end{equation}
\begin{equation}\label{eq:polar2}
    \begin{split}
       \int_{\mathbb{R}^N} \int_0^\infty  \psi_t*f(\mathbf x)\overline{\psi_t*g(\mathbf x)}\frac{dt}{t}\, dw(\mathbf x)
       &=\int_{\mathbb{R}^N} \mathcal F f(\xi)\overline{\mathcal F g(\xi)}\tilde c_\psi(\xi)\, dw(\xi).
    \end{split}
    \end{equation}

    \section{Proofs of Theorems \ref{teo:main1}, \ref{teo:main2}, and Corollary \ref{coro:S_t}}
    We start by proving Theorem \ref{teo:main1}. To this end, by \eqref{eq:dom1} and \eqref{eq:dom2}, it suffices to establish  that for every $1<p<\infty$  and $\alpha \in R$ the square functions $S_{\nabla_k,\phi}$, $k(\alpha)S_{K^{\{\alpha\}}}$, and $S_{\psi}$ are bounded on $L^p(dw)$. The $L^2(dw)$-bounds of the square functions are guaranteed by \eqref{eq:globaL2} and~\eqref{eq:globalL2_1}. {To finish the proof of Theorem \ref{teo:main1}  it suffices to check that the associated kernels  $t \nabla_{k,\mathbf x} \phi_t(\mathbf x,\mathbf y)$,  $k(\alpha)K^{\{\alpha\}}(t,\mathbf x,\mathbf y)$, and $K_{\psi}(t,\mathbf x,\mathbf y)=\psi_t(\mathbf x,\mathbf y)$  satisfy \eqref{eq:estimate_K11}--\eqref{eq:Holder_K33} and then apply Theorem \ref{teo:square-vector}. But these are guaranteed by Corollary \ref{coro:nabla_k}, Proposition \ref{propo:Kernel_K}, and Proposition \ref{propo:est_phi}.}

We now turn to prove Theorem \ref{teo:main2}. We start by verifying   \eqref{eq:lower_phi}. Proposition~\ref{prop:ord_implies_Dunkl} implies that $\mathcal F\phi\in C^{ \lfloor \mathbf N \rfloor +1}(\mathbb R^N)$ and
$$ |\partial^{\beta'}\mathcal F\phi(\xi)|\leq C_{\beta'} (1+\|\xi\|)^{-2s}\quad \text{ for } |\beta'|\leq  \lfloor \mathbf N \rfloor +1.$$
Thus it is easy to see that the function
$c_\phi$ (defined by \eqref{eq:c_phi}) is $C^{\lfloor \mathbf N \rfloor +1}$ away from the origin  and  homogeneous of degree zero.

Recall that by our assumption,  $\mathcal F\phi$
 is not identically zero along any direction (see \eqref{eq:non-degenarate}). {Hence, there is a constant $C>0$ such that
  $0<C^{-1}\leq c_\phi(\xi)\leq C$ for all $\xi \neq 0$.}
 Now, Theorem 1.2 of \cite{DzH} asserts that for every $1<q<\infty$,  the Dunkl multiplier operator
 $$ f\mapsto \mathcal T_{c_\phi}f:=\mathcal F^{-1}(c_{\phi}(\xi)\mathcal Ff(\xi)), $$
 initially defined on $L^q(dw)\cap L^2(dw)$, is bounded on $L^{q}(dw)$, invertible on $L^q(dw)$, and its inverse
 is of the form $f\mapsto \mathcal T_{1/c_\phi}f$. Let $f\in L^p(dw)\cap L^2(dw)$.
 Using the Plancherel identity~\eqref{eq:Plancherel}, we get
 \begin{equation}\label{eq:Planch8}
     \begin{split}
         \| f\|_{L^p(dw)}&=\sup_{g\in\mathcal S(\mathbb R^N),\; \|g\|_{L^{p'}(dw)}\leq 1} \Big|\int_{\mathbb{R}^N}f(\mathbf x)\overline{g(\mathbf x)}\, dw(\mathbf x)\Big|\\
         &= \sup_{g\in\mathcal S(\mathbb R^N),\; \|g\|_{L^{p'}(dw)}\leq 1}\Big|\int_{\mathbb{R}^N}\mathcal Ff(\xi)\overline{\mathcal Fg(\xi)}\, dw(\xi)\Big|\\
          &= \sup_{g\in\mathcal S(\mathbb R^N),\; \|g\|_{L^{p'}(dw)}\leq 1}\Big|\int_{\mathbb{R}^N}\mathcal Ff(\xi)\overline{\mathcal F(\mathcal T_{1/c_\phi}g)(\xi)}c_\phi(\xi)\, dw(\xi)\Big|.\\
     \end{split}
 \end{equation}
 Note that all the integrals are convergent, since all the functions $f$, $g$ and $\mathcal T_{1/c_\phi}g$ belong to $L^2(dw)$. From \eqref{eq:polar} and \eqref{eq:Planch8} we conclude
 \begin{equation*}
     \begin{split}
         \| f\|_{L^p(dw)}&=\sup_{g\in\mathcal S(\mathbb R^N),\; \|g\|_{L^{p'}(dw)}\leq 1}
          \int_{\mathbb{R}^N}\int_0^\infty t^2 \langle  \nabla_k(\phi_t*f)(\mathbf x),\nabla_k(\phi_t*\mathcal (\mathcal T_{1/c_\psi}g))(\mathbf x)\rangle \frac{dt}{t}\, dw(\mathbf x)\\
          &\leq \sup_{g\in\mathcal S(\mathbb R^N),\; \|g\|_{L^{p'}(dw)}\leq 1} \| S_{\nabla_k,\phi}(f)\|_{L^p(dw)} \| S_{\nabla_k,\phi}(\mathcal T_{1/c_\phi}g)\|_{L^{p'}(dw)} \\
          &\leq C_{p'}\sup_{g\in\mathcal S(\mathbb R^N),\; \|g\|_{L^{p'}(dw)}\leq 1} \| S_{\nabla_k,\phi}(f)\|_{L^p(dw)} \|\mathcal T_{1/c_\phi}g\|_{L^{p'}(dw)} \\
          &\leq C \|S_{\nabla_k,\phi}(f)\|_{L^p(dw)},
     \end{split}
 \end{equation*}
 which completes the proof of \eqref{eq:lower_phi} for $f\in L^p(dw)\cap L^2(dw)$. In order to relax the additional assumption $f\in L^2(dw)$, we apply the following easy approximation argument.  We take $f_n\in L^2(dw)\cap L^p(dw)$ such that $\lim_{n\to\infty} \|f-f_n\|_{L^p(dw)}=0$. Then
 \begin{equation*}\begin{split} \| f\|_{L^p(dw)}&=\lim_{n\to\infty} \| f_n\|_{L^p(dw)}\leq C \limsup_{n\to\infty} \|S_{\nabla_k,\phi}(f_n)\|_{L^p(dw)}\\
 &\leq C' \limsup_{n\to\infty} \|S_{\nabla_k,\phi}(f_n-f)\|_{L^p(dw)}+C'\|S_{\nabla_k,\phi}(f) \|_{L^p(dw)}
\leq C'\|S_{\nabla_k,\phi}(f) \|_{L^p(dw)},
 \end{split}\end{equation*}
 where in the last inequality we have  used Theorem~\ref{teo:main1}.

 The proof of \eqref{eq:lower_psi} is identical to that of \eqref{eq:lower_phi} and uses \eqref{eq:polar2}. Now \eqref{eq:lower_Gamma} follows from \eqref{eq:lower_phi}, since $S_{\nabla_k,\phi}f(x)\leq C  \mathfrak g_{\Gamma,\phi}f(\mathbf x)$, see \eqref{eq:dom1}.

Finally we prove Corollary \ref{coro:S_t}.
By direct calculations we have
\begin{equation*}
    \begin{split}
        \big(t\frac{d}{dt}\phi_{t}\big)(\mathbf{x})=-\mathbf{N}\phi_{t}(\mathbf{x})-\sum_{j=1}^{N}t^{-\mathbf{N}}\frac{x_{j}}{t}(\partial_j\phi)(\mathbf{x}/t)=\psi_{t}(\mathbf{x}),
    \end{split}
\end{equation*}
where
\begin{equation*}
    \psi(\mathbf{x})=-\mathbf{N}\phi(\mathbf{x})-\sum_{j=1}^{N}x_{j}(\partial_{j}\phi)(\mathbf{x}).
\end{equation*}
{Clearly, $\psi \in C^{2s} (\mathbb{R}^N)$ and satisfies \eqref{eq:main_assum}.} Moreover, by~\eqref{eq:integral_scaled}, we get
\begin{equation*}
    \int_{\mathbb{R}^N}\big(t\frac{d}{dt}\phi_{t}\big)(\mathbf{x})\,dw(\mathbf{x})=t\frac{d}{dt}\int_{\mathbb{R}^N}\phi_{t}(\mathbf{x})\,dw(\mathbf{x})=0,
\end{equation*}
so $\int_{\mathbb{R}^N}\psi(\mathbf{x})\,dw(\mathbf{x})=0$. Consequently,
\begin{equation}\label{eq:the_same}
    S_{\nabla_t,\phi}f(\mathbf x)=S_{\psi}f(\mathbf{x})
\end{equation}
and~\eqref{eq:S_t_upper} follows by Theorem~\ref{teo:main1}. To prove~\eqref{eq:S_t_lower}, we note that for any $t_1>0$ and $\xi \in \mathbb{R}^N$, $\xi \neq 0$, we have
\begin{equation*}
    \begin{split}
        \int_{t_1}^{\infty}(\mathcal{F}\psi)(t\xi)\,\frac{dt}{t}=\int_{t_1}^{\infty}(\mathcal{F}\psi_t)(\xi)\,\frac{dt}{t}=\int_{t_1}^{\infty}\frac{d}{dt}(\mathcal{F}\phi_t)(\xi)\,dt=-\mathcal{F}\phi(t_1\xi).
    \end{split}
\end{equation*}
Since $(\mathcal{F}\phi)$ satisfies \eqref{eq:non-degenarate}, there is $t_1>0$ such that $ \int_{t_1}^{\infty}(\mathcal{F}\psi)(t\xi)\,\frac{dt}{t} \neq 0$. So $(\mathcal{F}\psi)$ is not identically zero along  the direction of $\xi$. Thus~\eqref{eq:S_t_lower} follows from~\eqref{eq:the_same} and Theorem~\ref{teo:main2}.


\begin{thebibliography}{99}

\bibitem{ADzH}
J.-Ph. Anker, J. Dziuba\'nski, A. Hejna,
	\emph{Harmonic functions, conjugate harmonic functions and the Hardy space $H^1$ in the rational Dunkl setting}, J. Fourier Anal. Appl. 25 (2019), 2356--2418.

\bibitem{AH}
B. Amri, A. Hammi,
\textit{Dunkl-Schr\"odinger operators\/},
{Complex Anal. Oper. Theory (2018).}


\bibitem{deJeu}
	M.F.E. de Jeu,
	\emph{The Dunkl transform\/},
	Invent. Math. 113 (1993), 147--162.
	
\bibitem{RoeslerDeJeu}
	M. de Jeu, M. R\"osler,
	\emph{Asymptotic analysis for the Dunkl kernel},
	J. Approx. Theory 119 (2002), no. 1,
	110--126.

\bibitem{Dunkl0}
	C.F.~Dunkl,
	\emph{Reflection groups and orthogonal polynomials on the sphere},
 	Math. Z. 197 (1988), no. 1,
 	33--60.

\bibitem{Dunkl}
	C.F. Dunkl,
	\emph{Differential-difference operators associated to reflection groups\/},
	Trans. {Amer}. Math. 311 {(1989), no. 1,} 167--183{.}
	
\bibitem{Dunkl3}
	C.F.~Dunkl,
	\emph{Hankel transforms associated to finite reflection groups},
in: {\it Proc. of the special session on hypergeometric functions on domains of positivity, Jack polynomials and applications,}
	 Proceedings, Tampa 1991, Contemp. Math. 138 (1989),
	 123--138.

\bibitem{Dunkl2}
	C.F.~Dunkl,
	\emph{Integral kernels with reflection group invariance},
	Canad. J. Math. 43 (1991), no. 6,
	1213--1227.

\bibitem{Duo}
    J. Duoandikoetxea,
    \emph{Fourier Analysis},
    Graduate Studies in Mathematics, 29. American Mathematical Society, Providence, RI, 2001. xviii+222 pp. ISBN: 0-8218-2172-542-01.


\bibitem{DzH}
    J. Dziuba\'nski and A. Hejna,
    \emph{H\"ormander's multiplier theorem for the Dunkl transform},
     Journal of Functional Analysis 277 (2019), 2133-2159.

\bibitem{Graczyk}
    P. Graczyk, T. Luks, M. R\"osler,
    \emph{On the Green Function and Poisson Integrals of the Dunkl Laplacian},
    Potential Anal. 48 (2018), no. 3, 337–360.

	
\bibitem{Grafakos_classical}
	L. Grafakos,
	\emph{Classical Fourier Analysis},
	3rd edition, Graduate Texts in Mathematics, 249. Springer, New York,
	2014.
	
\bibitem{LiaoZhangLi2017}
    J. Liao, X. Zhang, Z. Li,
    \emph{On Littlewood-Paley functions associated with the Dunkl operator},
    Bull. Aust. Math. Soc. 96 (2017), no. 1, 126--138.
	
\bibitem{LiZhao}
    H. Li and M. Zhao,
    \emph{Square function estimates for Dunkl operators}, \href{https://arxiv.org/abs/2003.11843}{[arXiv:2003.11843]}.



\bibitem{Roesler2}
	M. R\"osler,
	\emph{Generalized Hermite polynomials and the heat equation for Dunkl operators\/},
	Comm. Math. Phys. {192} (1998), 519--542.

\bibitem{Roesle99}
	M. R\"osler,
	\emph{Positivity of Dunkl's intertwining operator\/},
	Duke Math. J. 98 (1999), no. 3, 445--463.

\bibitem{Roesler2003}
	M. R\"osler,
	\emph{A positive radial product formula for the Dunkl kernel\/},
	Trans. Amer.Math. Soc. 355 (2003), no. 6, 2413--2438{.}
	
\bibitem{Roesler3}
	M. R\"osler:
	\emph{Dunkl operators (theory and applications).
	In: Koelink, E., Van Assche, W. (eds.)
	Orthogonal polynomials and special functions} (Leuven, 2002), 93--135.
	Lect. Notes Math. 1817, Springer-Verlag (2003).
	
\bibitem{Roesler-Voit}
	M.~R\"osler, M. Voit,
	\emph{Dunkl theory, convolution algebras, and related Markov processes\/},
in \textit{Harmonic and stochastic analysis of Dunkl processes\/},
{P. Graczyk, M. R\"osler, M. Yor (eds.), 1--112, Travaux en cours 71,}
Hermann, Paris, 2008.

\bibitem{SoltaniJFA}
    F. Soltani,
    \emph{Littlewood-Paley operators associated with the Dunkl operator on $\mathbb{R}$},
     J. Funct. Anal. 221 (2005), no. 1,
     205--225.

\bibitem{SoltaniJIPAM}
    F. Soltani,
    \emph{Littlewood-Paley $g$-function in the Dunkl analysis on $\mathbb{R}^d$},
     JIPAM. J. Inequal. Pure Appl. Math. 6 (2005), no. 3, Article 84, 13 pp.

\bibitem{St1}
E.M. Stein,
\emph{Singular integral and differentiability properties of functions\/},
Princeton {Math.} Series 30, Princeton Univ. Press, 1970.

\bibitem{St2}
E.M. Stein,
\emph{Harmonic analysis\/}
{(\textit{real variable methods, orthogonality and oscillatory integrals\/})}, {Princeton Math. Series 43}, Princeton Univ. Press, 1993.
	
\bibitem{ThangaveluXu}
	S. Thangavelu, Y. Xu,
	\emph{Convolution operator and maximal function for the Dunkl transform},
	 J. Anal. Math. 97 (2005),
	 25--55.

\bibitem{Velicu1}
    A. Velicu,
    \emph{Sobolev-Type Inequalities for Dunkl Operators},
    J. Funct. Anal. 279 (2020), no. 7, 108695, 37 pp.	

\bibitem{Ch} Ch. Yacoub, personal communication.
\end{thebibliography}
\end{document}